\documentclass[graybox,envcountchap,sectrefs,amsart]{svmono}
\usepackage{amsmath} 
\usepackage{type1cm}
\usepackage{bpproc}
\usepackage{makeidx}   
\usepackage{graphicx}
\usepackage{multicol}
\usepackage[bottom]{footmisc}
\usepackage{newtxtext}  
\usepackage{newtxmath}
\usepackage{color}
\usepackage{hyperref}
\usepackage{tikz-cd}
\usepackage{rotating}
\usepackage[all]{xy}
\DeclareMathOperator{\Ker}{Ker}
\DeclareMathOperator{\depth}{depth}
\DeclareMathOperator{\pd}{pd}
\DeclareMathOperator{\tr}{tr}
\DeclareMathOperator{\Spec}{Spec}
\DeclareMathOperator{\T}{\mathbb{T}}
\DeclareMathOperator{\eigensystems}{eigensystems}
\DeclareMathOperator{\Ann}{Ann}
\DeclareMathOperator{\Supp}{Supp}
\DeclareMathOperator{\red}{red}

\makeindex
\graphicspath{{./}{images/}}

\begin{document}
\newcommand{\tensor}[1]{%
  \mathbin{\mathop{\otimes}\limits_{#1}}%
}
\title{Deformation Theory of Galois Representations and the Taylor--Wiles Method}
\author{Ehsan Shahoseini \thanks{The author was supported by a grant from IPM} \\
    \footnotesize{School of Mathematics, Institute for Research in Fundamental Sciences (IPM), Niavaran Sq., Tehran, Iran, P.O. Box: 19395-5746}\\
    \footnotesize{\href{mailto:shahoseini@ipm.ir}{shahoseini@ipm.ir}}\\
    \footnotesize{\href{mailto:ehsanshahoseini69@gmail.com}{ehsanshahoseini69@gmail.com}}}

%
\thispagestyle{empty}

\maketitle
\extrachap{Notations and Conventions}

$R\llbracket X_1,X_2,\ldots,X_n\rrbracket$ \quad ring of power series in $n$-variables with coefficients in the ring $R$ \\
$M(1)$ \quad  the Tate twist of the module $M$ \\
$p$ \quad  a fixed prime number \\
$\ell$ \quad  a prime number \\
$\Z$ \quad  the ring of integers \\
$\Z_{\ell}$ \quad  the ring of $\ell$-adic integers \\
$\Q$ \quad  the field of rational numbers \\
$\Q_{\ell}$ \quad  the field of $\ell$-adic numbers \\
$\varepsilon_p$ \quad  the $p$-adic cyclotomic character \\
$\overline{\varepsilon}_p$ \quad  the mod-$p$ cyclotomic character \\
$\mathbb{F}$ \quad  a fixed finite field of characteristic $p$ \\
$W(\mathbb{F})$ \quad  the ring of Witt vectors of $\mathbb{F}$ \\
$F$ \quad  a local field, a finite extension of $\mathbb{Q}_{\ell}$ for some prime number $\ell$ \\
$K$ \quad  a number field \\
$S$ \quad  a finite set of the places of a given number field $K$ \\
$\overline{L}$ \quad  a fixed algebraic closure of the field $L$ \\
$K_S$ \quad  maximal algebraic extension of the number field $K$ (in a fixed $\overline{K}$) unramified outside $S$ \\
$K_{\mathfrak{p}}$ \quad  completion of the number field $K$ at the prime ideal $\mathfrak{p}$ \\
$G_K=\Gal(\overline{K}/K)$ \quad  absolute Galois group of $K$ \\
$G_{K,S}=\Gal(K_S/K)$ \quad  Galois group of $K_S$ over $K$ \\
$G_{\mathfrak{p}}=\Gal(\overline{K}_{\mathfrak{p}}/K_{\mathfrak{p}})$ \quad  decomposition group at $\mathfrak{p}$, absolute Galois group of $K_{\mathfrak{p}}$ \\
$I_{\mathfrak{p}}$ \quad inertia group at $\mathfrak{p}$ \\
$\Frob{\mathfrak{p}}$ \quad Frobenius element at $\mathfrak{p}$ \\
$\tr(M)$ \quad  trace of the matrix $M$ \\
$\det(M)$ \quad  determinant of the matrix $M$ \\

\tableofcontents
\chapter{Deformation Theory of Galois Representations and the Taylor--Wiles Method}

\vspace{-3 cm}
\noindent
{\large{Ehsan Shahoseini}}

\label{intro}

\vspace{1 cm}

\abstract{In this chapter, we want to have an overview of the Taylor--Wiles patching method. For this purpose, at the first we recall Mazur's theory of deforming Galois representations and study both local and global deformation problems. Then, we go through the subject of Taylor-Wiles primes and examine the role that they play on the Galois side and the modular (automorphic) side. At the end, we arrive at the Taylor-Wiles patching method and use it to prove $R=\mathbb{T}$ in both minimal and non-minimal cases. Note that, in the Galois side we will work with totally real number fields, but for the modular side we will concentrate on $\mathbb{Q}$ to avoid difficulties of working with Hilbert modular forms.}

\vspace{1 cm}

Some references for this chaper are \cite{Allen}, \cite{CSS}, \cite{DDT}, \cite{Gee}, \cite{Kisin}, \cite{Shin}, \cite{Taylor2}, \cite{TW}, \cite{Wiles}.

\section{Deformation Theory of Galois Representations}
\label{sec:1}

Main references for this section are \cite{Allen} and \cite{Mazur2}.

\subsection{Galois Representations}

Throughout this chapter, let $p$ be a \textit{fixed} prime number, $\mathbb{F}$ be a finite field of characteristic $p$, $\ell$ be a prime number, $K$ be a number field, $S$ be a finite set of places of $K$, and $K_S$ be the maximal algebraic extension of $K$ (in a fixed algebraic closure $\overline{K}$ of $K$) unramified outside $S$. Also, let $K_{\mathfrak{p}}$ be the completion of $K$ at the prime ideal $\mathfrak{p}$. Put $G_K=\Gal(\overline{K}/K)$, $G_{K,S}=\Gal(K_S/K)$, and $G_{\mathfrak{p}}=\Gal(\overline{K}_\mathfrak{p}/\overline{K})$. Note that all Galois groups are profinite groups.

\begin{remark}
The group $G_\mathfrak{p}$ is topologically finitely generated, so $G_K$ is topologically (countably) infinitely generated. Note that we do not know if $G_{K,S}$ is topologically finitely generated or not.
\end{remark}

For deformation theory of Galois representations and its applications, we impose a weaker condition than (topologically) finite-generation:

\begin{definition} \label{def: p-finiteness condition}
Let $G$ be a profinite group. For the prime number $p$, we say that $G$ satisfies the $p$-finiteness (or $\Phi_p$-finiteness) condition, if for all open subgroups $G_0$ of $G$ we have $|\Hom{\text{cont}}(G_0, \mathbb{Z}/p\mathbb{Z})|< \infty$.
\end{definition}

The advantage of working with representations of groups that satisfy the $p$-finiteness condition is that in their deformation theory, universal deformation rings (which will be defined later) are always Noetherian.

\begin{example}
The groups $G_{\mathfrak{p}}$ and $G_{K,S}$ satisfy the $p$-finiteness conditon, but $G_K$ does not.
\end{example}

We have the following fundamental short exact sequence:
\begin{equation} \label{eq: local galois group exact sequence}
\{1\} \rightarrow I_{\mathfrak{p}} \rightarrow G_{\mathfrak{p}} \rightarrow \Gal(\overline{\mathbb{F}}/\mathbb{F}) \simeq \hat{\mathbb{Z}} \rightarrow \{1\}
\end{equation}
where $I_{\mathfrak{p}}$ is called the inertia subgroup at $\mathfrak{p}$, $\mathbb{F}$ is the residue field of $K_\mathfrak{p}$ which is a finite field with $q$ elements and $\Gal(\overline{\mathbb{F}}/\mathbb{F})$ is (topologically) cyclic and generated by the Frobenius element $\Frob{\mathfrak{p}}$ which sends $x$ to $x^q$. Note that under the isomorphism $\Gal(\overline{\mathbb{F}}/\mathbb{F}) \simeq \hat{\mathbb{Z}}$, we assume that $\Frob{\mathfrak{p}}$ goes to $1$.

\begin{example} \label{ex: package G_K,S}
For each prime ideal $\mathfrak{p}$, we have a continuous group homomorphism $G_\mathfrak{p} \rightarrow G_K$ which depends on the choice of the embedding $K \hookrightarrow K_{\mathfrak{p}}$ and thus is well-defined only up to conjugation (by an element of $G_K$). So, we get a continuous group homomorphism $G_\mathfrak{p} \rightarrow G_{K,S}$ which is again well-defined up to conjugation (by an element of $G_{K,S}$). Now, let $\mathfrak{p} \notin S$. Then, the map $G_\mathfrak{p} \rightarrow G_{K,S}$ factors though $I_\mathfrak{p}$, i.e. we get
\begin{equation*}
G_\mathfrak{p}/I_\mathfrak{p} \simeq \Gal(\overline{\mathbb{F}}/\mathbb{F}) \simeq \hat{\mathbb{Z}} \rightarrow G_{K,S}.
\end{equation*}
The above map is well-defined up to conjugation, too. Hence, the image of $\Frob{\mathfrak{p}} \in \Gal(\overline{\mathbb{F}}/\mathbb{F})$ defines, not an element but, a conjugacy class in $G_{K,S}$ which we denote it again by $\Frob{\mathfrak{p}}$ and call it the $\mathfrak{p}$-Frobenius conjugacy class. So, for all $\mathfrak{p}$ not in $S$, we get the package
\begin{equation}
\{G_{K,S}; \{\Frob{\mathfrak{p}}\}_{\mathfrak{p} \notin S}\}.
\end{equation} 
One of the main goals of algebraic number theory is the study of, not only $G_{K,S}$ but also, the whole above package.
\end{example}

\begin{remark}
The abelianization of the above package, i.e. $\{G^{ab}_{K,S}; \Frob{\mathfrak{p}}, \mathfrak{p} \notin S\}$ is well understood by class field theory. Note that since $G^{ab}_{K,S}$ is abelian, $\Frob{\mathfrak{p}}$ is an element of $G^{ab}_{K,S}$.
\end{remark}

Since the package $\{G_{K,S}; \{\Frob{\mathfrak{p}}\}_{\mathfrak{p} \notin S}\}$ is well-defined only up to conjugation, it is not possible to study it canonically. But, there is an approach; the Tannakian approach:
\begin{center}
Try to understand not the group itself, but its representations $G_{K,S} \rightarrow GL_n$.
\end{center}

But, $GL_n$ of what?

As $G_{K,S}$ is profinite, we like that $GL_n(-)$ be profinite, too.

\begin{definition} \label{def: coefficient ring}
For a fixed prime number $p$ and a fixed finite field $\mathbb{F}$ of characteristic $p$, by a coefficient ring, we mean a complete Noetherian local ring $A$ with residue field $\mathbb{F}$ (for a local ring $A$, we denote its unique maximal ideal by $\mathfrak{m}_A$). We denote the category of coefficient rings with fixed residue field $\mathbb{F}$ by \textbf{CNL}. A homomorphism in \textbf{CNL} is a continuous local homomorphism which is compatible with the induced isomorphism on the residue fields. Let \textbf{Art} be the full subcategory of the \textit{Artinian} objects of \textbf{CNL}. For a given $\Lambda \in \textbf{CNL}$, We let $\textbf{CNL}_{\Lambda}$ be the full subcategory of \textbf{CNL} of $\Lambda$-algebras and $\textbf{Art}_{\Lambda}$ be the full subcategory of Artinian objects of $\textbf{CNL}_{\Lambda}$.
\end{definition}

Note that for a coefficient ring $A$, $A$ and hence $GL_n(A)$ are profinite.

\begin{remark}
Let $W(\mathbb{F})$ be the ring of \textit{Witt vectors} of $\mathbb{F}$, i.e. the ring of integers of the unique unramified extension of $\mathbb{Q}_p$ with residue field $\mathbb{F}$. Then, for $A \in \textbf{CNL}$, we have a \textbf{CNL}-morphism $W(\mathbb{F}) \rightarrow A$, and in fact \textbf{CNL}=$\textbf{CNL}_{W(\mathbb{F})}$ and \textbf{Art}=$\textbf{Art}_{W(\mathbb{F})}$.
\end{remark}

Note that a $\Lambda$-algebra coefficient ring can be written as a quotient of $\Lambda\llbracket X_1,\cdots,X_n \rrbracket$, for some $n$.

\begin{definition} \label{def: residual representation}
Let $A \in \textbf{CNL}$ and let $\rho: G_{K,S} \rightarrow GL_n(A)$ be a representation. The reduction map $\pi: A \rightarrow A/\mathfrak{m}_A \simeq \mathbb{F}$ induces a reduction map $GL_n(A) \rightarrow GL_n(\mathbb{F})$ which we denote it again by $\pi$. We call $\overline{\rho}:=\pi \circ \rho$ the residual representation attached to $\rho$:

\begin{center}
\begin{tikzcd}
G_{K,S} \arrow[rd, "\overline{\rho}"] \arrow[r, "\rho"] & GL_n(A) \arrow[d, "\pi"] \\ 
& GL_n(\mathbb{F})
\end{tikzcd}.
\end{center}
\end{definition}

The following proposition shows that residually absolutely irreducible representations $\rho$ are determined, up to conjugation, via the trace of $\rho$:

\begin{proposition} \label{prop: residually irreducible representations}
Let $\rho:G \rightarrow GL_n(A)$ is a residually absolutely irreducible representations and $\rho^{\prime}:G \rightarrow GL_n(A)$ is another representation. If for all $h \in G$ we have $\tr(\rho(h))=\tr(\rho^{\prime}(h))$, then $\rho=g\rho^{\prime}g^{-1}$ for some $g \in GL_n(A)$.
\end{proposition}

By using \textit{Chebotarev density theorem}, we get the following corollary:

\begin{corollary} \label{cor: chebotarev and residually irreducible representation}
Let $\rho, \rho^{\prime}:G_{K,S} \rightarrow GL_n(A)$ are two representations and $\rho$ is residually absolutely irreducible. If for all $\mathfrak{p} \notin S$ we have $\tr(\rho(\Frob{\mathfrak{p}}))=\tr(\rho^{\prime}(\Frob{\mathfrak{p}}))$, then $\rho=g\rho^{\prime}g^{-1}$ for some $g \in GL_n(A)$. Also, we can assume that $\mathfrak{p}$ running through a set of prime ideals outside $S$ which has \textit{Dirichlet density} $1$.
\end{corollary} 

\subsection{Deforming Galois Representations}

\subsubsection{Universal Deformation Ring}

Let $G$ be a group and fix a continuous representation $\overline{\rho}:G \rightarrow GL_n(\mathbb{F})$.

\begin{definition} \label{def: lift and deformation}
For a ring $A \in \textbf{CNL}$, a \textit{lift} or a \textit{framed deformation} of $\overline{\rho}$ to $A$ is a continuous homomorphism $\rho:G \rightarrow GL_n(A)$ such that $ \rho \equiv \rho^{\prime} \pmod{\mathfrak{m}_A} $:
\begin{equation*}
\begin{tikzcd}
G \arrow[rd, "\overline{\rho}"] \arrow[r, "\rho"] & GL_n(A) \arrow[d, "\mod{\mathfrak{m}_A}"] \\ 
& GL_n(\mathbb{F})
\end{tikzcd}.
\end{equation*}
We say that two lifts $\rho$ and $\rho^{\prime}$ of $\overline{\rho}$ to $A$ are strictly equivalent if there exists $g \in 1+M_n(\mathfrak{m}_A)=\Ker(GL_n(A) \rightarrow GL_n(\mathbb{F}))$ such that $\rho=g\rho^{\prime}g^{-1}$. A \textit{deformaton} of $\overline{\rho}$ to $A$ is a strict equivalence class of lifts.
\end{definition}

\begin{remark}
We will often abuse the notation by denoting a deformation by a lift in its strict equivalence class.
\end{remark}

Let \textbf{SET} denote the category of \textit{Sets}.

\begin{definition} \label{def: lifting and deformation functor}
The \textit{lifting functor} or \textit{framed deformation functor} for $\overline{\rho}$ is the functor:
\begin{align*}
D_{\overline{\rho}}^{\square}: & \textbf{CNL} \rightarrow \textbf{SET} \\
& A \mapsto \text{\{lifts of} \: \overline{\rho} \: \text{to} \: A\}
\end{align*}
and the \textit{deformation functor} for $\overline{\rho}$ is the functor:
\begin{align*}
D_{\overline{\rho}}: & \textbf{CNL} \rightarrow \textbf{SET} \\
& A \mapsto \text{\{deformations of} \: \overline{\rho} \: \text{to} \: A\}.
\end{align*}
\end{definition}

\begin{remark}
We write $D_{\overline{\rho}, \Lambda}^{\square}$ and $D_{\overline{\rho}, \Lambda}$ for reductions of $D_{\overline{\rho}}^{\square}$ and $D_{\overline{\rho}}$ to $\textbf{CNL}_\Lambda$, respectively. Note that sometimes we will omit $\overline{\rho}$ and/or $\Lambda$ from the notation, if they are understood.
\end{remark}

\begin{definition}
We call a functor $F:\textbf{CNL} \rightarrow \textbf{SET}$ a continuous functor if for any $A \in \textbf{CNL}$, the natural map $F(A) \rightarrow \varprojlim F(A/\mathfrak{m}_A^i)$ be a bijection.
\end{definition}

\begin{proposition}
The functors $D_{\overline{\rho}}^{\square}$ and $D_{\overline{\rho}}$ are continuous.
\end{proposition}

\begin{corollary}
The functors $D_{\overline{\rho}}^{\square}$ and $D_{\overline{\rho}}$ are completely determined by their restriction to \textbf{Art}.
\end{corollary}

Let us recall that a functor $F:\textbf{CNL} \rightarrow \textbf{SET}$ is \textit{representable} if there exists a ring $R \in \textbf{CNL}$ and an isomorphism of functors $F \simeq \Hom{\textbf{CNL}}(R, -)$. If $F$ be representable by $R$, then there exists a universal object $\alpha^{univ} \in F(R)$ corresponding to $identity \in \Hom{\textbf{CNL}}(R, R) \simeq F(R)$ with the following property:
\begin{itemize}
\item[] For any $A \in \textbf{CNL}$ and any $\beta \in F(A)$, there is a unique \textbf{CNL}-morphism $f:R \rightarrow A$ such that $\beta=F(f)(\alpha^{univ})$.
\end{itemize}

For the lifting functor and the deformation functor we have the following important theorem:

\begin{theorem} \label{th: representability of liftng and deformation}
Let we have $\overline{\rho}:G \rightarrow GL_n(\mathbb{F})$ and let $G$ satisfies the $p$-finiteness condition.
\begin{itemize}
\item[(1)] (Kisin \cite{Kisin}) The functor $D_{\overline{\rho}}^{\square}$ is representable. We call the representing ring of it the \textit{universal lifting ring} or \textit{universal framed deformation ring} and denote it by $R_{\overline{\rho}}^{\square}$.
\item[(2)] (Mazur \cite{Mazur1}, Ramakrishna \cite{Ramakrishna}) If $\End_{\mathbb{F}[G]}(\overline{\rho})=\mathbb{F}$ (this is hold, for example when $\overline{\rho}$ is absolutely irreducible), then $D_{\overline{\rho}}$ is representable. We call the representing ring of it the \textit{universal deformation ring} and denote it by $R_{\overline{\rho}}$.
\end{itemize}
\end{theorem}

Hence, there exists a universal representation $\rho^{univ}: G \rightarrow GL_n(R_{\overline{\rho}})$ such that for each $A \in \textbf{CNL}$, every deformation $\rho: G \rightarrow GL_n(A)$ of $\overline{\rho}$ comes from a unique ring homomorphism $f: R_{\overline{\rho}} \rightarrow A$:
\begin{equation*}
\begin{tikzcd}
G \arrow[rd, "\rho"] \arrow[r, "\rho^{univ}"] & GL_n(R_{\overline{\rho}}) \arrow[d, "f"] \\
& GL_n(A)
\end{tikzcd}.
\end{equation*}

\begin{remark}
The universal deformaton $\rho^{univ}$ is in fact the corresponding element to $identity \in \Hom{\textbf{CNL}}(R_{\overline{\rho}}, R_{\overline{\rho}})$ in the correspondence $D_{\overline{\rho}}(R_{\overline{\rho}}) \cong \Hom{\textbf{CNL}}(R_{\overline{\rho}}, R_{\overline{\rho}})$ of Sets.
\end{remark}

It is trivial that the previous paragraph and Remark remain true for $D_{\overline{\rho}}^{\square}$ and $R_{\overline{\rho}}^{\square}$.

\begin{remark}
Let $\Lambda \in \textbf{CNL}$. Then, the restriction of $D_{\overline{\rho}}$ ($D_{\overline{\rho}}^{\square}$) to $\textbf{CNL}_{\Lambda}$, $D_{\overline{\rho}, \Lambda}$ ($D_{\overline{\rho}, \Lambda}^{\square}$), is representable by $R_{\overline{\rho}} \hat{\tensor{W(\mathbb{F})}} \Lambda$ ($R_{\overline{\rho}}^{\square} \hat{\tensor{W(\mathbb{F})}} \Lambda$), where $\hat{\otimes}$ means the \textit{completed tensor product}.
\end{remark}

\begin{remark}
First part of Theorem \ref{th: representability of liftng and deformation} can be proved by an explicit construction of $R_{\overline{\rho}}^{\square}$. To prove the second part of it there are (at least) four methods: proof of \textit{Mazur} \cite{Mazur1} using \textit{Schlessinger's} criterion, proof of \textit{Kisin} \cite{Kisin} using the quotient of $D_{\overline{\rho}}^{\square}$ by the action of smooth formal group $\hat{PGL_n}$, and two explicit constructions of $R_{\overline{\rho}}$ due to \textit{de Smit-Lenstra} \cite{DL} and \textit{Faltings} \cite{DDT}.
\end{remark}

We want to mention Schlessinger's criterion. Before it, let us recall a theorem of \textit{Grothendieck} about representability of functors. Note that, if $F: \textbf{CNL} \rightarrow \textbf{SET}$ be representable by $R \in \textbf{CNL}$ and the maps $A \rightarrow C$ and $B \rightarrow C$ be morphisms in \textbf{Art}, then the natural map $F(A \times_C B) \rightarrow F(A) \times_{F(C)} F(B)$ is a bijection, because:
\begin{align*}
F(A \times_C B)= & \Hom{\textbf{CNL}}(R, A \times_C B) \\
& = \Hom{\textbf{CNL}}(R, A) \times_{\Hom{\textbf{CNL}}(R, C)} \Hom{\textbf{CNL}}(R, B) \\
& = F(A) \times_{F(C)} F(B).
\end{align*}
Note that the second equality is the universal property of the fiber product.

Grothendieck showed that the converse of above is (almost) true. Before stating the Grothendieck theorem, we need the following definition:

\begin{definition} \label{def: tangent space}
Let $\mathbb{F}[\epsilon]:=\mathbb{F}[X]/\langle X^2 \rangle $ be the ring of dual numbers over $\mathbb{F}$. The tangent space of the functor $F: \textbf{CNL} \rightarrow \textbf{SET}$ is defined to be $F(\mathbb{F}[\epsilon])$. This tangent space is just a set. If we assume that for the maps $A \rightarrow C$ and $B \rightarrow C$ in \textbf{Art}, the natural map $F(A \times_C B) \rightarrow F(A) \times_{F(C)} F(B)$ is a bijection and also $F(\mathbb{F})$ is a singleton, then we can make this tangent space into an $\mathbb{F}$-vector space. First, we can define the addition on $F(\mathbb{F}[\epsilon])$ as follows:
\begin{equation*}
\begin{tikzcd}
F(\mathbb{F}[\epsilon]) \times F(\mathbb{F}[\epsilon]) = F(\mathbb{F}[\epsilon]) \times_{F(\mathbb{F})} F(\mathbb{F}[\epsilon]) = 
F(\mathbb{F}[\epsilon] \times_{\mathbb{F}} \mathbb{F}[\epsilon])  \arrow[r, "\psi"] & F(\mathbb{F}[\epsilon])
\end{tikzcd}
\end{equation*}
where $\psi (F(a+b\epsilon, a+c\epsilon))=F(a+(b+c)\epsilon)$. Then, we define the scalar multiplication of $\gamma \in \mathbb{F}$ on $F(\mathbb{F}[\epsilon])$ via $\gamma . F(a+b\epsilon):=F(a+\gamma b \epsilon)$. In fact, the map $\mathbb{F} \times F(\mathbb{F}(\epsilon)) \rightarrow F(\mathbb{F}(\epsilon))$ which determines the scalar multiplication, is induced by the map $\mathbb{F} \times \mathbb{F}(\epsilon) \rightarrow \mathbb{F}(\epsilon)$ which sends $(\gamma, \epsilon)$ to $\gamma \epsilon$.
\end{definition}

Now, we are ready to state Grothendieck's theorem:

\begin{theorem}{(Grothendieck)} \label{th: Grothendieck rep}
Let $F: \textbf{CNL} \rightarrow \textbf{SET}$ be a continuous functor such that $F(\mathbb{F})$ is a singleton. Then, $F$ is representable if and only if the following conditions hold:
\begin{itemize}
\item[(1)] For all maps $A \rightarrow C$ and $B \rightarrow C$ in \textbf{Art}, the natural map $F(A \times_C B) \rightarrow F(A) \times_{F(C)} F(B)$ is a bijection.
\item[(2)] $\dim_{\mathbb{F}}F(\mathbb{F}[\epsilon]) < \infty$ .
\end{itemize}
\end{theorem}

In practice, we can say that it is almost impossible to check the first condition of Grothendieck's Theorem \ref{th: Grothendieck rep} for all maps in \textbf{Art}. Schlessinger showed that for proving representability of $F: \textbf{Art} \rightarrow \textbf{SET}$, it is enough to check the first condition of Grothendieck's Theorem \ref{th: Grothendieck rep} for very restricted classes of maps in \textbf{Art}.

\begin{definition}
We say that a homomorphism $A \rightarrow C$ in \textbf{Art} is \textit{small}, if it is surjective and its kernel is principal and annihilated by $\mathfrak{m}_A$.
\end{definition}

\begin{theorem}{(Schlessinger's Criterion \cite{Schlessinger})} \label{th: Schlessinger}
Let $F: \textbf{CNL} \rightarrow \textbf{SET}$ be a continuous functor such that $F(\mathbb{F})$ is a singleton. For $\alpha: A \rightarrow C$ and $\beta: B \rightarrow C$ in \textbf{Art}, consider $f: F(A \times_C B) \rightarrow F(A) \times_{F(C)} F(B)$. Then, $F$ is representable if and only if the following conditions are satisfied:
\begin{itemize}
\item[(H1)] \: \: If $\alpha$ is small, then $f$ is surjective.
\item[(H2)] \: \: If $A=\mathbb{F}[\epsilon]$ and $C=\mathbb{F}$, then $f$ is bijective.
\item[(H3)] \: \: $\dim_{\mathbb{F}}F(\mathbb{F}[\epsilon]) < \infty$ .
\item[(H4)] \: \: If $A=B$ and $\alpha=\beta$ is small, then $f$ is bijective.
\end{itemize}
\end{theorem}

\begin{remark}
\begin{itemize}
\item[(1)] The functor $F: \textbf{CNL} \rightarrow \textbf{SET}$ is called \textit{nearly representable}, if it satisfies the first three conditions of Schlessinger's criterion. Note that if $F$ is nearly representable, then its tangent space $F(\mathbb{F}[\epsilon])$ has a natural $\mathbb{F}$-vector space structure.
\item[(2)] Let $G$ satisfy the $p$-finiteness condition and let $\overline{\rho}: G \rightarrow GL_n(\mathbb{F})$ be a representation. Then, $D_{\overline{\rho}}$ is nearly representable.
\end{itemize}
\end{remark}

Now, we are ready to study an example briefly:

\begin{example} \label{ex: lifting of unramified character}
Let $p \neq \ell$ be two prime numbers and $F$ be a finite extension of $\mathbb{Q}_{\ell}$ with residue field $k_F$ such that $| k_F | = q =\ell^m$ for some positive integer $m$. Also, let $\mathcal{O}$ be the ring of integers of some finite extension of $\mathbb{Q}_p$ with residue field $\mathbb{F}$ and let $\overline{\rho}: G_F \rightarrow \mathbb{F}^{\times}$ be an unramified character. Hence, $\overline{\rho}$ factors through $G_{k_F}$ which we denote it again by $\overline{\rho}$:
\begin{equation*}
\begin{tikzcd}
G_F \arrow[rd, "\overline{\rho}"] \arrow[r] & G_{k_F} \arrow[d, "\overline{\rho}"] \\
& \mathbb{F}^{\times}
\end{tikzcd}.
\end{equation*}
Since $G_{k_F}$ is topologically generated by \textit{Frobenius element} $\phi$, $\overline{\rho}: G_{k_F} \rightarrow \mathbb{F}^{\times}$ is completely determind by $\overline{\rho}(\phi)=\overline{a} \in \mathbb{F}^{\times}$. Let us lift $\overline{a}$ to an arbitrary $a \in \mathcal{O}^{\times}$. Consider the framed deformation functor:
\begin{align*}
D_{\overline{\rho}, \mathcal{O}}^{\square}: & \textbf{CNL}_{\mathcal{O}} \rightarrow \textbf{SET} \\
& A \mapsto \text{\{lifts of} \: \overline{\rho} \: \text{to} \: A \: \in \textbf{CNL}_{\mathcal{O}}\}.
\end{align*}
Note that in the definition of $D_{\overline{\rho}, \mathcal{O}}^{\square}$, the domain of $\overline{\rho}$ is all of $G_F$, since lifts of $\overline{\rho}$ may have ramifications and do not factor through $G_{k_F}$.
\

Let $P_F$ denote the \textit{wild inertia} subgroup of $G_F$. Then, since $\overline{\rho}$ is unramified, $\rho(P_F) \subseteq 1+ \mathfrak{m}_A$ for any lift $\rho$ of $\overline{\rho}$. But $P_F$ is a pro-$\ell$-group and $1+ \mathfrak{m}_A$ is a pro-$p$-group, and as $p \neq \ell$, $\rho(P_F)$ is trivial. So, any lift $\rho$ of $\overline{\rho}$ factors through $G_F/P_F$ which we denote it again by $\rho$:
\begin{equation*}
\begin{tikzcd}
G_F \arrow[rd, "\rho"] \arrow[r] & G_F/P_F \arrow[d, "\rho"] \\
& A
\end{tikzcd}.
\end{equation*}
Recall that $G_F/P_F \simeq \hat{\mathbb{Z}}^{(\ell)}(1) \rtimes \hat{\mathbb{Z}}$ is the Galois group of the maximal tame extension of $F$ (Note that $M(1)$ is the \textit{Tate twist} of $M$, and $\hat{\mathbb{Z}}^{(\ell)}=\prod \mathbb{Z}_{p^{\prime}}$ where $p^{\prime}$ runs over all prime numbers different from $\ell$). Let $\hat{\mathbb{Z}}^{(\ell)}$ be (topologically) generated by (the image of) $\tau$ which is called the tame generator. Then, we have $\phi \tau \phi^{-1}=\tau^q$. Let $R_{\overline{\rho}}^{\square}$ be the universal lifting ring and $\rho^{univ}$ be the universal lift. Since $G_F/P_F$ has two generators $\phi$ and $\tau$ with the relation $\phi \tau \phi^{-1}=\tau^q$, we should have $R_{\overline{\rho}}^{\square}=\mathcal{O}\llbracket X,Y\rrbracket /J$, where the ideal $J$ consists of relations. Thus, we get:
\begin{align*}
\rho^{univ}: & G_F/P_F \rightarrow (R_{\overline{\rho}}^{\square})^{\times}=(\mathcal{O}\llbracket X,Y\rrbracket /J)^{\times} \\
& \phi \mapsto a+X \\
& \tau \mapsto 1+Y
\end{align*}
and since $\phi \tau \phi^{-1}=\tau^q$, we get $(1+Y)^{q-1}=1$. Now, we have two cases:
\begin{itemize}
\item[(1)] $p \nmid q-1$: in this case $Y=0$ and $R_{\overline{\rho}}^{\square}=\mathcal{O}\llbracket X\rrbracket$.
\item[(2)] $p^{t} \parallel q-1$: in this case $R_{\overline{\rho}}^{\square}=\mathcal{O}[\mathbb{Z}/p^{t}\mathbb{Z}]\llbracket X\rrbracket$.
\end{itemize}
\end{example}

\subsubsection{Tangent Space}

We have seen the definition of the tangent space of the (framed) deformation functor, \eqref{def: tangent space}. Now, we want to interpret it in terms of \textit{group cohomology}.

\begin{definition}
Let $\overline{\rho}: G \rightarrow GL_n(\mathbb{F}) $ be a representation. Let $ad(\overline{\rho})$ denote $M_n(\mathbb{F})$ with the adjoint $G$-action, i.e. for $\sigma \in G$ and $M \in ad(\overline{\rho})$, we have $\sigma . M = \overline{\rho}(\sigma) M \overline{\rho}(\sigma)^{-1}$.
\end{definition}

Let $Z^1(G, ad(\overline{\rho}))$ denote the space of $1$-cocycles with coefficients in $ad(\overline{\rho})$.

\begin{proposition} \label{prop: tangent space and cohomology}
Let $\mathbb{F}$ be a finite extension of $\mathbb{F}_p$ and $\overline{\rho}: G \rightarrow GL_n(\mathbb{F}) $ be a representation. For the tangent spaces of $D_{\overline{\rho}}^{\square}$ and $D_{\overline{\rho}}$  we have the following:
\begin{itemize}
\item[(1)] $D_{\overline{\rho}}^{\square}(\mathbb{F}[\epsilon]) \simeq (\mathfrak{m}_{R_{\overline{\rho}}^{\square}}/(\mathfrak{m}^{2}_{R_{\overline{\rho}}^{\square}}, p))^{\vee} \simeq Z^1(G, ad(\overline{\rho}))$.
\item[(2)] If $D_{\overline{\rho}}$ be representable, then $D_{\overline{\rho}}(\mathbb{F}[\epsilon]) \simeq (\mathfrak{m}_{R_{\overline{\rho}}}/(\mathfrak{m}^{2}_{R_{\overline{\rho}}}, p))^{\vee} \simeq H^1(G, ad(\overline{\rho}))$.
\end{itemize} 
\end{proposition}

\begin{corollary}
If $G$ satisfies the $p$-finiteness condition, then $D_{\overline{\rho}}(\mathbb{F}[\epsilon])$ is a finite dimensional $\mathbb{F}$-vector space.
\end{corollary}

\begin{proposition} \label{prop: generator and relation for universal deformation ring}
Let $D_{\overline{\rho}}$ be representable. Also, let $r=\dim_{\mathbb{F}}H^1(G, ad(\overline{\rho}))$ and $s=\dim_{\mathbb{F}}H^2(G, ad(\overline{\rho}))$. Then, $R_{\overline{\rho}} \simeq W(\mathbb{F})\llbracket X_1, \cdots , X_r\rrbracket /(f_1, \cdots , f_s)$, where $f_i$ are power series in $W(\mathbb{F})\llbracket X_1, \cdots , X_r\rrbracket$.
\end{proposition}

\begin{definition}
If $H^2(G, ad(\overline{\rho}))=0$ (by the above Proposition, it is equivalent to say that $R_{\overline{\rho}}$ is a formal power series ring), we say that the deformation problem is unobstructed.
\end{definition}

\begin{conjecture}{(Mazur \cite{Mazur1})}
Let $K$ be a number field and $S$ be a finite set of places of $K$ containing all places above $p$ and $\infty$ (recall that $p$ is the characteristic of $\mathbb{F}$).
Let $\overline{\rho}: G_{K,S} \rightarrow GL_n({\mathbb{F}})$ be absolutely irreducible (thus $D_{\overline{\rho}}$ is representable).
Then, for $h_i=\dim_{\mathbb{F}}H^i(G_{K,S}, ad(\overline{\rho}))$, we have $\dim R_{\overline{\rho}}=1+h_1-h_2$ (note that the inequality $\dim R_{\overline{\rho}} \geq 1+h_1-h_2$ follows from Proposition \ref{prop: generator and relation for universal deformation ring}).
\end{conjecture}

\subsubsection{Deformation Conditions}

We fix a representation $\overline{\rho}: G \rightarrow GL_n(\mathbb{F})$, like before. We also fix a $\Lambda \in \textbf{CNL}$ and usually assume that $\Lambda=\mathcal{O}$, where $\mathcal{O}$ is the ring of integers of some finite totally ramified extension of $W(\mathbb{F})[1/p]$. Here we are interested to study subfunctors of $D_{\overline{\rho}}$ and $D_{\overline{\rho}}^{\square}$ consisting of deformations or lifts subject to certain conditions.

\subsubsection*{Fixed Determinant Condition}

Let us fix a continuous character $\chi: G \rightarrow \mathcal{O}^{\times}$ that $ \chi \equiv \det (\overline{\rho}) \pmod{\mathfrak{m}_{\mathcal{O}}} $. Let $D_{\overline{\rho}}^{\square, \chi} \subseteq D_{\overline{\rho}}^{\square}: \textbf{CNL}_{\mathcal{O}} \rightarrow \textbf{SET}$ be the subfunctor of lifts of $\overline{\rho}$ with $\det=\chi$, i.e. $\rho \in D_{\overline{\rho}}^{\square}(A)$ is an element of $D_{\overline{\rho}}^{\square, \chi}(A)$ if and only if $\det \rho=\imath \circ \chi$ where $\imath: \mathcal{O} \rightarrow A$ is the structure map, or equivalently we have the following commutative diagram:
\begin{equation*}
\begin{tikzcd}
G \arrow[r, "\rho"] \arrow[d, "\chi"]
& GL_n(A) \arrow[d, "\det"] \\
\mathcal{O}^{\times} \arrow[r, "\imath"]
& A^{\times}
\end{tikzcd}.
\end{equation*}
\\

This condition is stable under conjugation by elements of $1+M_n(\mathfrak{m}_A)$, hence we also get a subfunctor $D_{\overline{\rho}}^{\chi} \subseteq D_{\overline{\rho}}: \textbf{CNL}_{\mathcal{O}} \rightarrow \textbf{SET}$.
\\

For the representability of the above subfunctors we have:

\begin{proposition} \label{prop: representability of fixed determinant condition}
\begin{itemize}
\item[(1)] The subfunctor $D_{\overline{\rho}}^{\square, \chi} \subseteq D_{\overline{\rho}}^{\square}$ is representable by a quotient $R_{\overline{\rho}}^{\square, \chi}$ of $R_{\overline{\rho}}^{\square}$.
\item[(2)] If $D_{\overline{\rho}}$ be representable, then the subfunctor $D_{\overline{\rho}}^{\chi} \subseteq D_{\overline{\rho}}$ is representable by a quotient $R_{\overline{\rho}}^{\chi}$ of $R_{\overline{\rho}}$.
\end{itemize}
\end{proposition}

Let $ad^0(\overline{\rho}) \subseteq ad(\overline{\rho})$ denote the subset of matrices with trace $0$. Then, for the tangent space of the above subfunctors we have:

\begin{proposition} \label{prop: tangent space of fixed determinant condition}
\begin{itemize}
\item[(1)] $D_{\overline{\rho}}^{\square, \chi}(\mathbb{F}[\epsilon]) \simeq (\mathfrak{m}_{R_{\overline{\rho}}^{\square, \chi}}/(\mathfrak{m}^{2}_{R_{\overline{\rho}}^{\square, \chi}}, p))^{\vee} \simeq Z^1(G, ad^0(\overline{\rho}))$.
\item[(2)] If $D_{\overline{\rho}}$ (and hence $D_{\overline{\rho}}^{\chi})$) be representable, then $D_{\overline{\rho}}^{\chi}(\mathbb{F}[\epsilon]) \simeq (\mathfrak{m}_{R_{\overline{\rho}}^{\chi}}/(\mathfrak{m}^{2}_{R_{\overline{\rho}}^{\chi}}, p))^{\vee} \simeq \im(Z^1(G, ad^0(\overline{\rho})) \rightarrow H^1(G,ad(\overline{\rho}))) \simeq H^1(G, ad^0(\overline{\rho}))$.
\end{itemize}
\end{proposition}

\begin{definition} \label{def: deformation condition}
By a \textit{deformation condition} (or \textit{deformation problem}), we mean a collection $\mathcal{D}$ of lifts $(A, \rho)$ to objects $A \in \textbf{CNL}_{\Lambda}$ satisfying the following properties:
\begin{itemize}
\item[(1)] $(\mathbb{F}, \overline{\rho}) \in \mathcal{D}$.
\item[(2)] If $(A, \rho) \in \mathcal{D}$ and $f: A \rightarrow B$ be a morphism in $\textbf{CNL}_{\Lambda}$, then $(B, f \circ \rho) \in \mathcal{D}$.
\item[(3)] If $A \rightarrow C$ and $B \rightarrow C$ be morphisms in $\textbf{Art}_{\Lambda}$ and if $(A, \rho_A)$
and $(B, \rho_B)$ are elements of $\mathcal{D}$, then $(A \times_C B, \rho_A \times \rho_B) \in \mathcal{D}$.
\item[(4)] If $(A_i, \rho_i)$ is an inverse system of elements of $\mathcal{D}$ and $\varprojlim A_i \in \textbf{CNL}_{\Lambda}$, then $(\varprojlim A_i, \varprojlim \rho_i) \in \mathcal{D}$.
\item[(5)] The collection $\mathcal{D}$ is closed under strict equivalence.
\item[(6)] If $A \hookrightarrow B$ in an injection in $\textbf{CNL}_{\Lambda}$ and $(A, \rho)$ is a lift such that $(B, f \circ \rho) \in \mathcal{D}$, then $(A, \rho) \in \mathcal{D}$.
\end{itemize}
\end{definition}

\begin{proposition} \label{prop: deformation conditions arise in this way}
Let $R_{\overline{\rho}, \Lambda}^{\square} \twoheadrightarrow R$ be a surjection in $\textbf{CNL}_{\Lambda}$ satisfying the following property:
\begin{itemize}
\item[(\textbf{P})] For any lift $\rho: G \rightarrow GL_n(A)$ ($A \in \textbf{CNL}_{\Lambda})$) and any $g \in 1+M_n(\mathfrak{m}_A)$, the map $R_{\overline{\rho}, \Lambda}^{\square} \rightarrow A$ induced by $\rho$ factors through $R$ if and only if the map induced by $g \rho g^{-1}$ factors through $R$.
\end{itemize}
Then, the collection of lifts factor through $R$ form a deformation condition. Moreover, every deformation condition arises in this way.
\end{proposition}

\subsubsection{Local Deformation Conditions} 
\label{subsub: local deformation}

Since we will work only with $GL_2$, from now on we restrict ourselves to this case.

\subsubsection*{Ordinary Case}

Let $F$ be a finite extension  of $\mathbb{Q}_p$. Assume that for $\overline{\rho}: G_F \rightarrow GL_2(\mathbb{F})$ we have
$\overline{\rho} = \begin{bmatrix}
\overline{\chi}_1 & *\\
0 & \overline{\chi}_2
\end{bmatrix}$,
where $\overline{\chi}_i: G_F \rightarrow \mathbb{F}^{\times}$ are continuous characters. Also, we denote the \textit{inertia subgroup} of $G_F$ by $I$ and we let $\overline{\rho}(I) \neq 1$ and $\overline{\chi}_1(I)=1$. Fix a continuous character $\delta: I \rightarrow \mathcal{O}^{\times}$. Consider the functor $D_{\overline{\rho}}^{ord}: \textbf{CNL}_{\mathcal{O}} \rightarrow \textbf{SET}$ such that:
\begin{align*}
D_{\overline{\rho}}^{ord}(A) =\text{\{lifts} \: \rho \: & \text{of} \: \overline{\rho} \: \text{to} \: A \in \: \textbf{CNL}_{\mathcal{O}} \: \text{such that} \: \rho \: \text{is strictly equivalent to} \\ &
\begin{bmatrix}
\chi_1 & *\\
0 & \chi_2
\end{bmatrix}
\text{with} \: \chi_{1}\mid_I=1 \: \text{and} \: \chi_{2}\mid_I=\delta \}.
\end{align*}
Then, $D_{\overline{\rho}}^{ord}$ is a deformation condition, called the \textit{ordinary deformation} of $\overline{\rho}$.

\begin{proposition} \label{prop: D^ord is representable and its representating ring}
\begin{itemize}
\item[(1)] $D_{\overline{\rho}}^{ord}$ is representable by a ring $R_{\overline{\rho}}^{ord} \in \textbf{CNL}_{\mathcal{O}}$, which is a quotient of $R_{\overline{\rho}}^{\square}$.
\item[(2)] We have $R_{\overline{\rho}}^{ord} \simeq \mathcal{O}\llbracket X_1, \cdots ,X_r\rrbracket$ with $r=4+[F:\mathbb{Q}_p]$.
\end{itemize}
\end{proposition}

\subsubsection*{Minimal Case}

Let $F$ be a finite extension of $\mathbb{Q}_{\ell}$ with $\ell \neq p$ and suppose we have a representation $\overline{\rho}: G_F \rightarrow GL_2(\mathbb{F})$. Also, again by $I$ we mean the inertia subgroup of $G_F$.

\begin{itemize}
\item[(i)] Let $1 \neq \overline{\rho}(I) \subseteq \Biggl\{
\begin{bmatrix}
1 & *\\
0 & 1
\end{bmatrix}
 \Biggl\}$. Let consider the functor $D_{\overline{\rho}}^{min}: \textbf{CNL}_{\mathcal{O}} \rightarrow \textbf{SET}$ such that:
 \begin{align*}
D_{\overline{\rho}}^{min}(A) =\text{\{lifts} \: \rho \: \text{of} \: \overline{\rho} \: \text{to} \: & A \in \: \textbf{CNL}_{\mathcal{O}} \: \text{such that} \: \rho(I) \: \text{is strictly equivalent to} \\ 
& \text{a subgroup of}
\Biggl\{ \begin{bmatrix}
1 & *\\
0 & 1
\end{bmatrix} \Bigg\} \}.
\end{align*}
\item[(ii)] Let $\overline{\rho}=
\begin{bmatrix}
\overline{\chi}_1 & 0\\
0 & \overline{\chi}_2
\end{bmatrix}
=\overline{\chi}_1 \oplus \overline{\chi}_2$ with $\overline{\chi}_1\mid_I=1$ and $\overline{\chi}_2\mid_I\neq 1$. Now, consider the functor $D_{\overline{\rho}}^{min}: \textbf{CNL}_{\mathcal{O}} \rightarrow \textbf{SET}$ such that:
\begin{align*}
D_{\overline{\rho}}^{min}(A) =\text{\{lifts} \: \rho \: & \text{of} \: \overline{\rho} \: \text{to}
 \: A \in \: \textbf{CNL}_{\mathcal{O}} \: \text{such that} \: \rho \: \text{is strictly equivalent to} \\
& \: \chi_1 \oplus \chi_2 \: \text{with} \: \chi_{1}\mid_I=1 \: \text{and} \: \chi_{2}\mid_I=* \}
\end{align*}
where
\begin{tikzcd}
*=I \arrow[r, "\overline{\chi}_2"] & \mathbb{F}^{\times} \arrow[r, "\text{Teich}"] & \mathcal{O}^{\times} \arrow[r] & A^{\times}
\end{tikzcd}.
Note that the middle map is the \textit{Teichmuller lift} and the third one is the structure map of $A$ as $\mathcal{O}$-algebra.
\end{itemize}
In both cases, $D_{\overline{\rho}}^{min}$ is a deformation condition and such deformations are called \textit{minimally ramified deformations} of $\overline{\rho}$.

\begin{proposition} \label{prop: D^min is representable and its representating ring}
For both above cases, we have:
\begin{itemize}
\item[(1)] $D_{\overline{\rho}}^{min}$ is representable by a ring $R_{\overline{\rho}}^{min} \in \textbf{CNL}_{\mathcal{O}}$.
\item[(2)] $R_{\overline{\rho}}^{min} \simeq \mathcal{O}\llbracket X_1, X_2, X_3, X_4\rrbracket$.
\end{itemize}
\end{proposition}

\begin{itemize}
\item[(iii)] \: More generally, if $\overline{\rho}(I)$ has prime to $p$ order, then there is a functor $D_{\overline{\rho}}^{min}: \textbf{CNL}_{\mathcal{O}} \rightarrow \textbf{SET}$ for which:
\begin{align*}
D_{\overline{\rho}}^{min}(A) =\text{\{lifts} \: \rho \: \text{of} \: \overline{\rho} \: \text{to}
 \: A \in \: \textbf{CNL}_{\mathcal{O}} \: \text{such that} \: \rho(I) \rightarrow \overline{\rho}(I) \: \text{is an isomorphism} \}
\end{align*}
where the map $\rho(I) \rightarrow \overline{\rho}(I)$ is the $(\text{mod} \: \mathfrak{m}_A)$ map.
In this case, agian $D_{\overline{\rho}}^{min}$ is a deformation condition which is, again, called the minimally ramified deformation of $\overline{\rho}$.
\end{itemize}

\begin{remark}
One can also add the fixed determinant condition to \textit{local deformation conditions} and get the deformation functors $D_{\overline{\rho}, \mathcal{O}}^{ord, \chi}$ and $D_{\overline{\rho}, \mathcal{O}}^{min, \chi}$ (where $\chi: G_F \rightarrow \mathcal{O}^{\times}$ is a continuous character). 
\end{remark}

\subsubsection{Global Deformation Conditions} 
\label{subsub: global deformation}

Fix a number field $K$, an odd prime number $p$, and a finite set $S$ of primes of $K$ containing all primes above $p$. Let $K_S$ be the maximal algebraic extension of $K$ that is unramified outside $S \cup \{\text{infinite places of} \: K \}$ and put $G_{K,S}=\Gal(K_S/K)$. Let $\mathcal{O}$ be the ring of integers of a finite extension of $\mathbb{Q}_p$ and put $\mathbb{F}=\mathcal{O}/\mathfrak{m}_{\mathcal{O}}$. Also, fix a continuous representation $\overline{\rho}: G_{K,S} \rightarrow GL_2({\mathbb{F}})$ (note that we can work with $GL_n$ such that $p \nmid 2n$, but for our applications it is enough to work with $GL_2$).

We have a deformation functor $D_{\overline{\rho}}: \textbf{CNL}_{\mathcal{O}} \rightarrow \textbf{SET}$ such that if we have $\End_{\mathbb{F}[G_{K,S}]}(\overline{\rho})=\mathbb{F}$, it is representable by a ring $R_{\overline{\rho}} \in \textbf{CNL}_{\mathcal{O}}$. We want to impose some (determinant and local) conditions on $D_{\overline{\rho}}$. Note that for any place $v$ of $K$, we get a map of functors $D_{\overline{\rho}} \rightarrow D_{\overline{\rho}\mid_{G_{K_v}}}$ which sends $\rho$ to $\rho\mid_{G_{K_v}}$.
\\

Now, let us fix the following data:
\begin{itemize}
\item[($*$)] A continuous character $\chi: G_{K,S} \rightarrow \mathcal{O}^{\times}$.
\item[($**$)] \: For each $v \in S$, a deformation condition $D_v \subseteq D_{\overline{\rho}\mid_{G_{K_v}}}^{\square}$ (in fact, $D_v \subseteq D_{\overline{\rho}\mid_{G_{K_v}}}^{\square, \chi} \subseteq D_{\overline{\rho}\mid_{G_{K_v}}}^{\square}$).
\end{itemize}
Using the above data, we define the tuple $\mathcal{S}=(\overline{\rho}, S, \chi, \mathcal{O}, \{D_v \}_{v \in S})$ and will refer to it as a \textit{global deformation condition}.

\begin{definition}
We say that a lift $\rho$ of $\overline{\rho}$ to $A$ is of type $\mathcal{S}$ if:
\begin{itemize}
\item[(1)] $\rho$ is unramified outside $S$.
\item[(2)] $\det \rho=\chi$.
\item[(3)] $\rho \mid_{G_{K_v}} \in D_v(A)$ for all $v \in S$.
\end{itemize}
A deformation of $\overline{\rho}$ is of type $\mathcal{S}$ if one (and equivalently all) lifts in its strict equivalence class is of type $\mathcal{S}$.
\end{definition}

Now, consider the functor $D_{\mathcal{S}}: \textbf{CNL}_{\mathcal{O}} \rightarrow \textbf{SET}$ defined by:
\begin{align*}
D_{\mathcal{S}}(A) = \{\text{deformations of} \: \overline{\rho} \: \text{to} \: A \in \textbf{CNL}_{\mathcal{O}} \: \text{of type} \: \mathcal{S} \}.
\end{align*}

\begin{proposition} \label{prop: D_S is representable}
If $\End_{\mathbb{F}[G_{K,S}]}(\overline{\rho})=\mathbb{F}$, the functor $D_{\mathcal{S}}$ is representable by a quotient $R_{\mathcal{S}}$ of $R_{\overline{\rho}}$.
\end{proposition}

In fact, we know that the fixed determinant condition is representable by a quotient $R_{\overline{\rho}}^{\chi}$ of $R_{\overline{\rho}}$. Let $R_v$ be the quotient of $R_{\overline{\rho} \mid_{G_{K_v}}}^{\square}$ representing $D_v$.
Put $R_{S}^{\square}:= \hat{\tensor{\mathcal{O}, v \in S}} R_{\overline{\rho} \mid_{G_{K_v}}}^{\square}$ and $R_{\mathcal{S}}^{loc}:= \hat{\tensor{\mathcal{O}, v \in S}} R_v$.
Then, $D_{\mathcal{S}}$ is represented by $R_{\overline{\rho}}^{\chi} \hat{\tensor{R_{S}^{\square}}} R_{\mathcal{S}}^{loc}$.

\begin{definition} \label{def: T-framed lift and deformation}
\begin{itemize}
\item[(1)] Fix $T \subseteq S$. A \textit{$T$-framed lift} of $\overline{\rho}$ to $A \in \textbf{CNL}_{\mathcal{O}}$ is a tuple $(\rho, \{\beta_v \}_{v \in T})$, where $\rho$ is a lift of $\overline{\rho}$ to $A$ and $\beta_v$ is an element of $1+M_2(\mathfrak{m}_A)$ for all $v \in T$.
\item[(2)] We say that a $T$-framed lift $(\rho, \{\beta_v \}_{v \in T})$ is of type $\mathcal{S}$ if $\rho$ is.
\item[(3)] Two $T$-framed lifts $(\rho, \{\beta_v \}_{v \in T})$ and $(\rho^{\prime}, \{\beta_v^{\prime} \}_{v \in T})$ are strictly equivalent if there exists $g \in 1+M_2(\mathfrak{m}_A)$ such that $\rho^{\prime}=g \rho g^{-1}$ and $\beta_v^{\prime}=g \beta$ for all $v \in T$. A \textit{$T$-framed deformation} is a strict equivalence class of $T$-framed lifts.
\end{itemize}
\end{definition}

Let consider the functor $D_{\mathcal{S}, T}: \textbf{CNL}_{\mathcal{O}} \rightarrow \textbf{SET}$ for which:
\begin{align*}
D_{\mathcal{S}, T}(A)= \{\text{T-framed deformations of} \: \overline{\rho} \: \text{to} \: A \in \textbf{CNL}_{\mathcal{O}} \: \text{of type} \: \mathcal{S} \}.
\end{align*}

For representability of the above functor, we have:

\begin{proposition} \label{prop: D_S,T is representable}
\begin{itemize}
\item[(1)] If $\End_{\mathbb{F}[G_{K,S}]}(\overline{\rho})=\mathbb{F}$ or $T \neq \varnothing$, then the functor $D_{\mathcal{S}, T}$ is representable by a ring $R^T_{\mathcal{S}} \in \textbf{CNL}_{\mathcal{O}}$.
\item[(2)] If $\End_{\mathbb{F}[G_{K,S}]}(\overline{\rho})=\mathbb{F}$ and $T \neq \varnothing$ and $| T |=t$, then we have $R^T_{\mathcal{S}} \simeq R_{\mathcal{S}}\llbracket X_1, \ldots, X_{4t-1}\rrbracket$.
\end{itemize}
\end{proposition}

\subsubsection*{Relative Tangent Space for Global Deformation Conditions}

Let the $T$-framed lift $(\rho, \{\beta_v \}_{v \in T})$ be of type $\mathcal{S}$. Like before, let $R_v$ represent $D_v$ for $v \in S$ and $R_{\mathcal{S}, T}$ represent $D_{\mathcal{S}, T}$. Put $R_{\mathcal{S}}^{T-loc}:= \hat{\tensor{\mathcal{O}, v \in T}} R_v$. Then, we have that $R_{\mathcal{S}, T}$ has a canonical $R_{\mathcal{S}}^{T-loc}$-algebra structure. For the relative tangent space of $D_{\mathcal{S}, T}$, we have:

\begin{proposition} \label{prop: tangent space of D_S,T}
Put $\mathfrak{m}^T_{\mathcal{S}}:=Max(R^T_{\mathcal{S}})$ and $\mathfrak{m}^{T-loc}_S:=Max(R_{\mathcal{S}}^{T-loc})$. Then, we have:
\begin{equation*}
D_{\mathcal{S}, T}(\mathbb{F}[\epsilon])=(\mathfrak{m}^T_{\mathcal{S}}/((\mathfrak{m}^T_{\mathcal{S}})^2, \mathfrak{m}^{T-loc}_S))^{\vee}=H_{\mathcal{S}, T}^1(G_{K,S}, ad^0(\overline{\rho}))
\end{equation*}
where $H_{\mathcal{S}, T}^1$ is the first cohomology group of a somewhat complicated complex, whose definition we opt to omit.
\end{proposition}

\section{Taylor--Wiles Primes}
\label{sec:2}

The main reference for this section and next one is \cite{Allen}.
\subsection{Taylor--Wiles Primes, Galois Side}

Like before, we fix a global deformation condition $\mathcal{S}=(\overline{\rho},S,\chi,\mathcal{O},\{D_v\}_{v\in S})$ for a number field $K$. Recall that $\overline{\rho}:G_{K,S} \rightarrow GL_2(\mathbb{F})$ and $p$ is the characteristic of $\mathbb{F}$.

\begin{definition} \label{def: Taylor--Wiles primes}
A \textit{Taylor--Wiles prime}, for $\mathcal{S}$, is a prime $v$ of $K$ which is disjoint from $S$ and satisfies the following:
\begin{itemize}
\item[(i)] $q_v:=Nr(v) \equiv 1 \mod{p}$.
\item[(ii)] $\overline{\rho}(\Frob{v})$ has distinct $\mathbb{F}$-rational eigenvalues.
\end{itemize}
Moreover, we say that a Taylor--Wiles prime $v$ has \textit{level} $N$, if further we have $q_v \equiv 1 \mod{p^N}$ and $N$ is the biggest integer with this property.
\end{definition}

\begin{remark}
The second condition in the previous definition is not restrictive. In fact, if the eigenvalues of $\overline{\rho}(\Frob{v})$ not be $\mathbb{F}$-rational, they will be after a quadratic extension of $\mathbb{F}$.
\end{remark}

\begin{proposition} \label{prop: diagonal lift at TW primes}
Let $v$ be a Taylor--Wiles prime for $\mathcal{S}$. For any $A \in \textbf{CNL}_{\mathcal{O}}$ and any lift $\rho_v:G_{K_v} \rightarrow GL_2(A)$ of $\overline{\rho}|_{G_{K_v}}$, $\rho_v$ is conjugate to a diagonal lift $\begin{bmatrix} \chi_1 & 0\\
0 & \chi_2
\end{bmatrix}
=\chi_1 \oplus \chi_2$.
\end{proposition}

Let $v$ be a Taylor--Wiles prime for $\mathcal{S}$. Let $R_v^{\square, \chi}$ be the universal lifting ring for $\overline{\rho}|_{G_{K_v}}$ with fixed determinant $\chi$ and let $\rho_v^\chi$ be the universal lift.
By Proposition \ref{prop: diagonal lift at TW primes}, $\rho_v^\chi$ is conjugate to $\chi_1 \oplus \chi_2$ with $\chi_i:G_{K_v} \rightarrow (R_v^{\square, \chi})^\times$ and $\chi_1 \chi_2 =\chi$.
In particular, as $\chi$ is unramified at $v$, we have $\chi_1|_{I_{K_v}}=\chi_2|^{-1}_{I_{K_v}}$. Since $\overline{\rho}$ is 
unramified at $v$, $\chi_1|_{I_{K_v}}$ is a pro-$p$-character of $I_{K_v^{ab}/K_v} \simeq k_v^\times \times \mathbb{Z}_{\ell}^d \times$ (a finite $\ell$-group), where $k_v$ is the residue field of $K$ at $v$ (i.e. $k_v:=\mathcal{O}_{K_v}/\mathfrak{m}_{K_v}$, where $\mathfrak{m}_{K_v}$ is the unique maximal ideal of $\mathcal{O}_{K_v}$), $\ell$ is the characteristic of $k_v$, and $d=[K_v:\mathbb{Q}_p]$. Therefore, since $p \nmid v$, $\chi_1|_{I_{K_v}}$ factored through $k_v^\times$.
Let $\Delta_v$ be the maximal $p$-power quotient of $k_v^\times$, the ring $\mathcal{O}[\Delta_v]$ be the group algebra, and $\mathfrak{a}_v$ be the augmentation ideal; i.e. $\mathfrak{a}_v=\langle g-1: g\in \Delta_v \rangle$.
Note that $\chi_1|_{I_{K_v}}$ determines an $\mathcal{O}[\Delta_v]$-algebra structure on $R_v^{\square, \chi}$. Moreover, there exists a natural surjection $R_v^{\square, \chi} \twoheadrightarrow R_v^{\square, \chi, nr}$ with
 kernel $\mathfrak{a}_vR_v^{\square, \chi}$, where $R_v^{\square, \chi, nr}$ is the universal lifting ring of $\overline{\rho}|_{G_{K_v}}$ of lifts $\rho$ such that $\rho(I_{K_v})=1$ and $\det(\rho)=\chi$. Note that since $\chi_1|_{I_{K_v}}$ determines the action of $\Delta_v$ and $R_v^{\square, \chi, nr}$ is the universal lifting ring of $\overline{\rho}|_{G_{K_v}}$ of lifts which are unramified at $v$ (and have fixed determinant $\chi$), so the action of $\Delta_v$ on $R_v^{\square, \chi, nr}$ is trivial and thus the augmentation ideal acts as zero on $R_v^{\square, \chi, nr}$, hence the kernel of the map $R_v^{\square, \chi} \twoheadrightarrow R_v^{\square, \chi, nr}$ is given by the augmentation ideal.
\\

Now, let $Q$ be a finite set of Taylor--Wiles primes. Also, let $\Delta_Q= \prod_{v \in Q} \Delta_v$, the ring $\mathcal{O}[\Delta_Q]$ is the group algebra and $\mathfrak{a}_Q$ is the augmentation ideal.
Then, we can define the (augmented) deformation condition $\mathcal{S}_Q=(\overline{\rho}, S \cup Q, \chi, \mathcal{O}, \{D_v\}_{v \in S} \cup \{D_v^\chi\}_{v \in Q})$, 
where for $v \in Q$, $D_v^\chi$ is the deformation condition of all lifts of $\overline{\rho}|_{G_{K_v}}$ with $\det =\chi|_{G_{K_v}}$. Then, by assuming $\End_{\mathbb{F}[G_{K,S}]}(\overline{\rho})=\mathbb{F}$, our
 new deformation problem is also representable and hence we get the universal deformation rings $R_{\mathcal{S}}$ and $R_{\mathcal{S}_Q}$, and also the $T$-ftamed universal deformation rings $R^T_{\mathcal{S}}$ and $R^T_{\mathcal{S}_Q}$ for any $T \subseteq S$.
 Then, $R^T_{\mathcal{S}_Q}$ has an $\mathcal{O}[\Delta_Q]$-algebra structure, and the natural surjection $R^T_{\mathcal{S}_Q} \twoheadrightarrow R^T_{\mathcal{S}}$ has kernel $\mathfrak{a}_QR^T_{\mathcal{S}_Q}$.
\\

Recall that for $T \subseteq S$, the tangent space of $R^T_{\mathcal{S}}$ is given by a cohomology group $H^1_{\mathcal{S}, T}(ad^0(\overline{\rho}))$, see Proposition \ref{prop: tangent space of D_S,T}. We denote the dimension of this cohomology group by $h^1_{\mathcal{S}, T}(ad^0(\overline{\rho}))$.

From now on, we assume that the following two conditions, along with two other technical conditions (concerning the dimensions of certain cohomology groups) which we do not state (as their statements are complicated), are hold:
\begin{itemize}
\item[(1)] $\overline{\rho}|_{G_{K(\zeta_p)}}$ is absolutely irreducible, where $\zeta_p$ is a primitive $p$-th root of unity.
\item[(2)] $K$ is totally real and $\overline{\rho}$ is totally odd, i.e. $\det(c_w)=-1$ for all infinite places $w$ of $K$, where $c_w$ is the complex conjugation at $w$.
\end{itemize}

Also, let $H^1_{\mathcal{S}^\perp, T}(ad^0(\overline{\rho})(1))$  be a certain cohomology group, whose definition we omit (since it is rather technical), and let us denote its dimension by $h^1_{\mathcal{S}^\perp, T}(ad^0(\overline{\rho})(1))$.

Under the above assumptions, we have the following important numerology:
\begin{itemize}
\item[(1)] Minimal case: $T=\emptyset$. Then:
\begin{equation} \label{eq: diamond}
h^1_{\mathcal{S}}(ad^0(\overline{\rho})) = h^1_{\mathcal{S}^\perp}(ad^0(\overline{\rho})(1)).
\end{equation}
\item[(2)] Non-minimal case: $T \supseteq \{v: v|p\}$ (e.g. $T=S$). Put $|T|=t$. Then:
\begin{equation*}
h^1_{\mathcal{S},T}(ad^0(\overline{\rho}))=t-1-[K:\mathbb{Q}] + h^1_{\mathcal{S}^\perp ,T}(ad^0(\overline{\rho})(1))
\end{equation*}
and since $dimR^{T-loc}_{\mathcal{S}}=1+3t+[K:\mathbb{Q}]$, it follows that:
\begin{equation} \label{eq: heart}
dimR^{T-loc}_{\mathcal{S}} + h^1_{\mathcal{S},T}(ad^0(\overline{\rho})) = h^1_{\mathcal{S}^\perp ,T}(ad^0(\overline{\rho})(1)) +4t.
\end{equation}
\end{itemize}

Let $Q$ be a finite set of Taylor--Wiles primes. As we saw, from the global deformation condition $\mathcal{S}$ we can define the (augmented) global deformation condition $\mathcal{S}_Q$. The main point is that the left hand side of the above formulas \ref{eq: diamond} and \ref{eq: heart} for $\mathcal{S}_Q$ only depends on $\mathcal{S}$.
\\

\begin{definition} \label{def: enormous subgroup}
Let $\Gamma$ be a subgroup of $SL_2(\mathbb{F})$ with absolutely irreducible action on $\mathbb{F}^2$ such that the eigenvalues of any $\gamma \in \Gamma$ are $\mathbb{F}$-rational. Let us denote the trace-zero subspace of $M_2(\mathbb{F})$ by $ad^0$ and consider it with adjoint $\Gamma$-action. We say that $\Gamma$ is enormous if it satisfies the following properties:
\begin{itemize}
\item[(1)] $\Gamma$ has no quotient of order $p$.
\item[(2)] $H^0(\Gamma,ad^0)=H^1(\Gamma,ad^0)=0$.
\item[(3)] For any simple $\mathbb{F}$-submodule $W$ of $ad^0$, there exists a $\gamma \in \Gamma$ with distinct eigenvalues such that $W^\gamma \neq 0$.
\end{itemize} 
\end{definition}

\begin{theorem} \label{th: classification of enormous subgroups}
If $\Gamma \subseteq GL_2(\mathbb{F})$ acts absolutely irreducibly and (as always) $p>2$, then $\Gamma$ will be enormous except in the following cases:
\begin{itemize}
\item[(1)] $p=3$ and image of $\Gamma$ in $PGL_2(\mathbb{F}_3)$ is conjugate to $PSL_2(\mathbb{F}_3)$.
\item[(2)] $p=5$ and image of $\Gamma$ in $PGL_2(\mathbb{F}_5)$ is conjugate to $PSL_2(\mathbb{F}_5)$.
\end{itemize}
\end{theorem}

\begin{proposition} \label{prop: existence of TW primes Q_N}
For a fixed global deformation condition $\mathcal{S}$, let $\Gamma=\overline{\rho}(G_{K(\zeta_p)})$ be enormous and put $q=h^1_{\mathcal{S}^\perp, T}(ad^0(\overline{\rho})(1))$. Then, for any positive integer $N$, there exists a (finite) set of Taylor--Wiles primes $Q_N$ of level $N$ such that:
\begin{itemize}
\item[(1)] $|Q_N|=q$.
\item[(2)] $H^1_{\mathcal{S}_{Q_N}^\perp, T}(ad^0(\overline{\rho})(1))=0$.
\end{itemize}
\end{proposition}

\begin{corollary} \label{cor: corollary of existence of TW primes Q_N}
There exists a non-negative integer $q$ such that for any positive integer $N$, there is a set of Taylor--Wiles primes $Q_N$ of level $N$ and of cardinality $q$, and a surjection $R_{\mathcal{S}}^{T-loc}\llbracket X_1, \cdots , X_g\rrbracket \rightarrow R_{\mathcal{S}_{Q_N}}^T$ where:
\begin{itemize}
\item[(i)] Minimal case ($T=\emptyset, R_{\mathcal{S}}^{T-loc}=\mathcal{O}$): $g=q$.
\item[(ii)] Non-minimal case ($T \supseteq \{v: v|p\}, |T|=t$): $\dim R_{\mathcal{S}}^{T-loc} +g = q +4t$.
\end{itemize}
(compare with formulas \eqref{eq: diamond}) and \eqref{eq: heart}.)
\end{corollary}

\begin{definition} \label{def: TW datum}
A \textit{Taylor--Wiles datum} ($Q, \{\alpha_v\}_{v \in Q}$) is a set $Q$ of Taylor--Wiles primes and \textit{a choice} $\alpha_v$ of an eigenvalue of $\overline{\rho}(\Frob{v})$, for each $v \in Q$.
\end{definition}

As we saw in Proposition \ref{prop: diagonal lift at TW primes} and the discussion after it, for the $\mathcal{S}_Q$-type universal deformation $\rho_{\mathcal{S}_Q}^{univ}:G_{K, S \cup Q} \rightarrow GL_2(R_{\mathcal{S}_Q})$ we
 have $\rho_{\mathcal{S}_Q}^{univ}|_{G_{K_V}} \simeq \chi_{v,1} \oplus \chi_{v,2}$ for any $v \in Q$ with $\chi_{v,i} \circ Art_{K_v}|_{\mathcal{O}_{K_v}^\times}:\mathcal{O}_{K_v}^\times \rightarrow R_{\mathcal{S}_Q}^\times$ factors through $\Delta_v$, where $Art_{K_v}$ is the \textit{local Artin map} in
  \textit{local class field theory}. The choice of an eigenvalue $\alpha_v$ of $\overline{\rho}(\Frob{v})$ determines an ordering between $\chi_{v,1}$ and $\chi_{v,2}$ by $\chi_{v,1}(\Frob{v})=\alpha_v$. Hence, a 
  Taylor-Wiles datum induces an $\mathcal{O}$-algebra map $\mathcal{O}[\Delta_Q] \rightarrow R_{\mathcal{S}_Q}$ by sending $\delta \in \Delta_v$ to $\chi_{v,1}(\delta)$, and thus 
  we get an $\mathcal{O}[\Delta_Q]$-algebra structure on $R_{S_Q}$. Also, the surjection $R_{\mathcal{S}_Q} \twoheadrightarrow R_{\mathcal{S}}$ has kernel $\mathfrak{a}_Q$.
\\

Now, from what we have seen in this section, and by letting $|Q|=q$, we have the following commutative diagrams:
\begin{itemize}
\item[(i)] Minimal case ($T=\emptyset$):

\begin{equation} \label{eq: minimal}
\begin{tikzcd}
& \mathcal{O}\llbracket \mathbb{Z}_p^q\rrbracket \arrow[ldd, dashrightarrow] \arrow[d, twoheadrightarrow] \\
& \mathcal{O}[\Delta_Q] \arrow[d] \\
\mathcal{O}\llbracket X_1, \cdots ,X_g\rrbracket \arrow[r, twoheadrightarrow]
& R_{\mathcal{S}_Q}
\end{tikzcd}
\end{equation}

such that, by using isomorphism $\mathcal{O}\llbracket \mathbb{Z}_p^q\rrbracket \simeq \mathcal{O}\llbracket Y_1, \cdots, Y_q\rrbracket =:S_\infty$, for the augmentation ideal $\mathfrak{a}_\infty=\langle Y_1, \cdots, Y_q \rangle \subseteq S_\infty$ we 
have $R_{\mathcal{S}_Q}/\mathfrak{a}_\infty R_{\mathcal{S}_Q} \simeq R_{\mathcal{S}}$; and if $Q$ be as in the Corollary \ref{cor: corollary of existence of TW primes Q_N}, then $g=q$.
\item[(ii)] Non-minimal case ($T \supseteq \{v:v|p\}, |T|=t$):

Fix an isomorphism $R^T_{\mathcal{S}_Q} \simeq R_{\mathcal{S}_Q} \hat{\tensor{\mathcal{O}}} \Omega$ with $\Omega=\mathcal{O}\llbracket Z_1, \cdots , Z_{4t-1}\rrbracket$.

\begin{equation} \label{eq: non-min}
\begin{tikzcd}
& \Omega\llbracket \mathbb{Z}_p^q\rrbracket \arrow[ldd, dashrightarrow] \arrow[d, twoheadrightarrow] \\
& \Omega[\Delta_Q] \arrow[d] \\
R^{T-loc}_{\mathcal{S}}\llbracket X_1, \cdots ,X_g\rrbracket \arrow[r, twoheadrightarrow]
& R^T_{\mathcal{S}_Q}
\end{tikzcd}
\end{equation}
\end{itemize}

such that, by using the isomorphism $\Omega\llbracket \mathbb{Z}_p^q\rrbracket \simeq \Omega\llbracket Y_1, \cdots, Y_q\rrbracket =:S_\infty$, for the 
augmentation ideal $\mathfrak{a}_\infty=\langle Z_1, \cdots Z_{4t-1}, Y_1, \cdots, Y_q \rangle \subseteq S_\infty$ we have $R^T_{\mathcal{S}_Q}/\mathfrak{a}_\infty R^T_{\mathcal{S}_Q} \simeq R^T_{\mathcal{S}}$; and 
if $Q$ be as in the Corollary \ref{cor: corollary of existence of TW primes Q_N}, then $\dim R_{\mathcal{S}}^{T-loc}\llbracket X_1, \cdots , X_g\rrbracket = \dim S_\infty$ (look at \ref{subsub: global deformation} for definitions of $R^T_{\mathcal{S}}$ and $R_{\mathcal{S}}^{T-loc}$).

\subsection{Taylor--Wiles primes, Modular Side}

Before going into the modular aspects of Taylor--Wiles primes, let's recall some background about Hecke algebras and Galois representations with values in Hecke algebras.

\subsubsection{Hecke Algebras}

From now on in this chapter, we assume $K=\mathbb{Q}$ to avoid working with \textit{Hilbert modular forms}. We can assume that our modular forms are of weight $k \geq 2$ and level $\Gamma=\Gamma_1(N)$, but 
for simplicity we assume $k=2$ which is enough for our purposes. Also, assuming $N \geq 4$ makes some simplicity in the proofs of some statements, if you want. Let $S \supset \{p,\infty\} \cup \{l : l | N\}$ be 
a finite set of primes of $\mathbb{Q}$. Also, let $S_2(\Gamma,R)$ denote the space of modular forms of weight $2$ and level $\Gamma$ with coefficients in the ring $R$; hence fixing an
 isomorphism $\imath:\mathbb{C} \rightarrow \overline{\mathbb{Q}}_p$ implies an isomorphism $S_2(\Gamma,\mathbb{C}) \simeq S_2(\Gamma, \overline{\mathbb{Q}}_p)$.

\begin{definition} \label{def: Hecke algebra}
For a finite set of primes $S$ of $\mathbb{Q}$, we define the (universal) \textit{Hecke algebra} as $\mathbb{T}^{S,univ}=\mathbb{T}^S:=\mathbb{Z}[T_\ell,S_\ell]_{\ell \in \Spec(\mathbb{Z}), \ell \notin S}$. For 
a ring $A$, we define $\mathbb{T}^S_A:=\mathbb{T}^S \tensor{\mathbb{Z}} A=A[T_\ell,S_\ell]_{\ell \in \Spec(\mathbb{Z}), \ell \notin S}$. Note that if there is no confusion and $A$ is known from the context, we 
just write $\mathbb{T}^S$ instead of $\mathbb{T}^S_A$. Also, for a $\T_A^S$-module $M$, we define $\T_A^S(M)=\T^S(M):=\im(\T_A^S \rightarrow \End_A(M))$.
\end{definition}

\begin{remark}
In the above definition, we consider $T_\ell$ and $S_\ell$ just as polynomial variables, but as elements of $\T_{\mathbb{C}}^S$, they act on $S_2(\Gamma,\mathbb{C})$ as the usual \textit{Hecke operator} $T_\ell$ and 
as the \textit{diamond operator} $\langle \ell \rangle$, respectively. Via this action, $S_2(\Gamma,\mathbb{C})$  is a semisimple $\T^S_{\mathbb{C}}$-module.
\end{remark}

\begin{definition} \label{def: Hecke eigenform and eigensystem}
We say that a modular form $f \in S_2(\Gamma,\mathbb{C})$ is a \textit{Hecke-eigenform} for $\T^S_{\mathbb{C}}$, if it is an eigenvector for all Hecke operators $\{T_\ell\}_{\ell \notin S}$. Let us denote the corresponding eigenvalue 
by $a_\ell$; hence $T_\ell f=a_\ell f$ for all primes $\ell \notin S$. By an eigensystem, we mean a (surjective) ring homomorphism $\lambda_f=\lambda:\T^S_{\mathbb{C}}(S_2(\Gamma,\mathbb{C})) \rightarrow \mathbb{C}$ such 
that $\lambda_f(T_\ell)=a_\ell$ for $\ell \notin S$, where $f \in S_2(\Gamma,\mathbb{C})$ is a Hecke-eigenform for $\T^S_{\mathbb{C}}$.
\end{definition}

Now, (since $S_2(\Gamma,\mathbb{C})$ is a semisimple $\T^S$-module) the \textit{Peterson inner product} implies that each $T_\ell$ is a normal operator and hence we get the 
decomposition $\T^S_{\mathbb{C}}(S_2(\Gamma,\mathbb{C})) \simeq \displaystyle \prod_{\eigensystems} \mathbb{C}$. Then, by the isomorphism $\imath: \mathbb{C} \rightarrow \overline{\mathbb{Q}}_p$, we 
get the decomposition $\T^S_{\overline{\mathbb{Q}}_p}(S_2(\Gamma,\overline{\mathbb{Q}}_p)) \simeq \displaystyle \prod_{\eigensystems} \overline{\mathbb{Q}}_p$. Also, for any
eigensystem $\lambda:\T^S_{\overline{\mathbb{Q}}_p}(S_2(\Gamma,\overline{\mathbb{Q}}_p)) \rightarrow \overline{\mathbb{Q}}_p$ (equivalently, for any Hecke-eigenform $f \in S_2(\Gamma,\overline{\mathbb{Q}}_p)$), we
have a Galois representation $\rho_\lambda=\rho_f:G_{\mathbb{Q},S} \rightarrow GL_2(\overline{\mathbb{Q}}_p)$ such that for any prime $\ell \notin S$, the characteristic polynomial of $\rho_\lambda(\Frob{\ell})$ is given 
by $X^2-\lambda(T_\ell)X+\ell \lambda(S_\ell)$. So, by gluing these, we get a Galois representation:
\begin{equation} \label{eq: non-integral Gal. rep.}
\rho:=\prod_\lambda \rho_\lambda:G_{\mathbb{Q},S} \rightarrow GL_2(\T^S_{\overline{\mathbb{Q}}_p}(S_2(\Gamma,\overline{\mathbb{Q}}_p)))
\end{equation}
suth that for any prime $\ell \notin S$, the characteristic polynomial of $\rho(\Frob{\ell})$ is given by $X^2-T_\ell X+\ell S_\ell$.

We seek for an integral version of this story. First, let us recall the \textit{Eichler--Shimura} isomorphism:

\begin{theorem}{(Eichler--Shimura)} \label{th: Eichler--Shimura}
There is an isomorphism of \: $\T^S$-modules $M_2(\Gamma,\mathbb{C}) \oplus \overline{S_2(\Gamma,\mathbb{C})} \simeq H^1(\Gamma,\mathbb{C})$, where $M_2(\Gamma,\mathbb{C})$ is the space of all modular forms of weight $2$ and level $\Gamma$.
\end{theorem}

Note that we have the isomorphism $H^1(\Gamma,\mathbb{C}) \simeq H^1(\Gamma,\mathbb{Z}) \tensor{\mathbb{Z}} \mathbb{C}$ as finite dimensional $\mathbb{C}$-vector spaces, and 
isomorphism $H^1(\Gamma,\mathcal{O}) \simeq H^1(\Gamma,\mathbb{Z}) \tensor{\mathbb{Z}} \mathcal{O}$ as finitely generated $\mathcal{O}$-modules. Also, we have the 
isomorphism $H^1(\Gamma,\mathcal{O}) \tensor{\mathcal{O}} \overline{\mathbb{Q}}_p \simeq H^1(\Gamma,\overline{\mathbb{Q}}_p) \simeq H^1(\Gamma,\mathbb{C})$ (which 
contains $S_2(\Gamma,\mathbb{C}) \simeq S_2(\Gamma,\overline{\mathbb{Q}}_p)$, by the Eichler--Shimura Theorem \ref{th: Eichler--Shimura}) as $\T^S_{\mathbb{C}} \simeq \T^S_{\overline{\mathbb{Q}}_p}$-modules.

Now, choose a Hecke-eigenform $f \in S_2(\Gamma,\overline{\mathbb{Q}}_p)$ and consider the composition map:
\begin{equation*}
\lambda_f:\T^S_{\overline{\mathbb{Q}}_p}(H^1(\Gamma,\overline{\mathbb{Q}}_p)) \rightarrow \T^S_{\overline{\mathbb{Q}}_p}(S_2(\Gamma,\overline{\mathbb{Q}}_p)) \rightarrow \overline{\mathbb{Q}}_p.
\end{equation*}
This map induces another map, which we call $\lambda_f$ again:
\begin{equation*}
\lambda_f:\T^S_{\mathcal{O}}(H^1(\Gamma,\mathcal{O})) \rightarrow \mathcal{O}.
\end{equation*}
Let $\mathcal{O}\rightarrow \mathbb{F}$ be the quotient map (by $\mathfrak{m}_{\mathcal{O}}$) and denote the composition of this quotient map and $\lambda_f$ by $\overline{\lambda}_f$. Also, 
let $\mathfrak{m}:=\Ker(\overline{\lambda}_f)$ which is a maximal ideal. Then, to this $\mathfrak{m}$ (equivalently, to $\overline{\lambda}_f)$) we can associate a Galois 
representation $\overline{\rho}_{\mathfrak{m}}:G_{\mathbb{Q},S} \rightarrow GL_2(\mathbb{F})$ such that for all primes $\ell \notin S$, the characteristic polynomial of $\overline{\rho}_{\mathfrak{m}}(\Frob{\ell})$ is 
given by $X^2-\overline{\lambda}_f(T_\ell)+\ell \overline{\lambda}_f(S_\ell)$, which is equal to $X^2-\lambda_f(T_\ell) X+\ell \lambda_f(S_\ell)$ modulo $\mathfrak{m}$.

\begin{definition} \label{def: non-Eisenstein ideal}
With the above notations, we say that the maximal ideal $\mathfrak{m}$ is \textit{non-Eisenstein}, if the residual representation $\overline{\rho}_{\mathfrak{m}}$ is absolutely irreducible.
\end{definition}

\begin{proposition} \label{prop: result of being non-Eisenstein}
If $\mathfrak{m}$ is non-Eisenstein, then $H^1(\Gamma,\mathcal{O})_{\mathfrak{m}}$ is a finite free $\mathcal{O}$-module. Also, since $\T^S_{\mathcal{O}}(H^1(\Gamma,\mathcal{O}))_{\mathfrak{m}} \subseteq \End_{\mathcal{O}}(H^1(\Gamma,\mathcal{O})_{\mathfrak{m}})$, we deduce that $\T^S_{\mathcal{O}}(H^1(\Gamma,\mathcal{O}))_{\mathfrak{m}}$ is $\mathcal{O}$-flat.
\end{proposition}

We have $\T^S_{\mathcal{O}}(H^1(\Gamma,\mathcal{O}))_{\mathfrak{m}} \hookrightarrow \T^S_{\mathcal{O}}(H^1(\Gamma,\mathcal{O}))_{\mathfrak{m}} \tensor{\mathcal{O}} \overline{\mathbb{Q}}_p \simeq \prod \overline{\mathbb{Q}}_p$, where the product is over all eigensystems above $\mathfrak{m}$. So, 
we get a Galois representation $\rho:G_{\mathbb{Q},S} \rightarrow GL_2(\T^S_{\mathcal{O}}(H^1(\Gamma,\mathcal{O}))_{\mathfrak{m}} \tensor{\mathcal{O}} \overline{\mathbb{Q}}_p)$ such that the characteristic polynomial 
of $\rho(\Frob{\ell})$ is $X^2-T_\ell X +\ell S_\ell$. This representation descends (by a theorem of \textit{Carayol}, which we do not state it here) to a representation as follows:
\begin{equation} \label{eq: integral Gal. rep.}
\rho_{\mathfrak{m}}:G_{\mathbb{Q},S} \rightarrow GL_2(\T^S_{\mathcal{O}}(H^1(\Gamma,\mathcal{O}))_{\mathfrak{m}}).
\end{equation}

\subsubsection{Back to: Taylor--Wiles Primes, Modular Side}

Let $\mathcal{O}$ and $\mathbb{F}$ be as before and $p>2$. Fix an absolutely irreducible representation $\overline{\rho}:G_{\mathbb{Q},S} \rightarrow GL_2(\mathbb{F})$. Assume that $\overline{\rho} \simeq \overline{\rho}_g$ for a Hecke-eigenform $g \in S_2(\Gamma,\mathcal{O})$ (equivalently, 
assume $\overline{\rho}$ has one modular lift, which is $\rho_g$ in fact), and assume $\Gamma_1(N) \leq \Gamma \leq \Gamma_0(N)$ such that $\{\ell:\ell |N\} \subseteq S$ and $\Gamma$ is torsion-free.

\begin{definition} \label{def: defining some subgroups}
For a finite set of Taylor--Wiles primes $Q$, the subgroups $\Gamma_1(Q) \leq \Gamma_Q \leq \Gamma_0(Q) \leq \Gamma$ defined as follows:
\begin{itemize}
\item[(i)] $\Gamma_0(Q):=\Gamma \cap \Gamma_0(\prod_{v \in Q} v)$.
\item[(ii)] $\Gamma_1(Q):=\Gamma \cap \Gamma_1(\prod_{v \in Q} v)$.
\item[(iii)] \: $\Gamma_Q$ is the kernel of the map $\Gamma_0(Q) \rightarrow \Xi$, where $\Xi$ is the maximal $p$-power quotient of $\Gamma_0(Q)/\Gamma_1(Q)$, which is isomorphic with $\prod_{v \in Q} (\mathbb{Z}/v\mathbb{Z})^\times$. So, $\Xi \simeq \Delta_Q$.
\end{itemize}
\end{definition}

Recall that $\T^S_{\mathcal{O}}=\T^S=\mathcal{O}[T_\ell,S_\ell]$. For a subset $\Sigma \subseteq S$, we also define $\T^{S, \Sigma}:=\T^S[\{U_v\}_{v \in \Sigma}]$. Note again that, $T_\ell$, $S_\ell$, and $U_v$ are 
just polynomial variables, but these universal Hecke algebras act on the spaces of modular forms, and on homology and cohomology of congruence subgroups and modular curves attached to them. 
Let $\T^S(\Gamma):=\T^S(H^1(\Gamma,\mathcal{O}))=\im(\T^S \rightarrow \End_{\mathcal{O}}(H^1(\Gamma,\mathcal{O})))$ and $\T^{S, \Sigma}(\Gamma):=\T^{S, \Sigma}(H^1(\Gamma,\mathcal{O}))=\im(\T^{S, \Sigma} \rightarrow \End_{\mathcal{O}}(H^1(\Gamma,\mathcal{O})))$. As we 
assumed $\overline{\rho} \simeq \overline{\rho}_g$ for a Hecke-eigenform $g \in S_2(\Gamma,\mathcal{O})$, we obtain a maximal ideal $\mathfrak{m}$ of $\T^S(\Gamma)$ which can be considered as a maximal 
ideal of $\T^S$, again denoted by $\mathfrak{m}$, in the support of $H^1(\Gamma,\mathcal{O})$. Now, consider the action of $\T^S(\Gamma)_{\mathfrak{m}}$ on $H^1(\Gamma,\mathcal{O})_\mathfrak{m} \simeq H^1(Y,\mathcal{O})_{\mathfrak{m}}$ for $Y=Y(\Gamma)=\Gamma \symbol{92} \mathbb{H}$. We 
have that for $i \neq 1$, $H^i(\Gamma,\mathbb{F})_{\mathfrak{m}}=0$ and hence $H^1(\Gamma,\mathcal{O})_{\mathfrak{m}} \simeq H^1(Y,\mathcal{O})_{\mathfrak{m}}$ is torsion-free. So, we have the 
duality $H^1(Y,\mathcal{O})_{\mathfrak{m}}=\Hom{\mathcal{O}}(H_1(Y,\mathcal{O})_{\mathfrak{m}},\mathcal{O})$ as $\T^S$-modules and transposition identifies $\T^S(\Gamma)_{\mathfrak{m}}$ with $\im(\T^S_{\mathfrak{m}} \rightarrow \End_{\mathcal{O}}(H_1(Y,\mathcal{O})_{\mathfrak{m}}))$.

Recall that we have a fixed Taylor--Wiles datum $(Q,\{\alpha_v\}_{v \in Q})$. We can pull back $\mathfrak{m} \subseteq \T^S$ to a maximal ideal of $\T^{S \cup Q}$ which we denote it again by $\mathfrak{m}$. Now, for 
each $v \in Q$, we have that $X^2-T_vX+vS_v \in \T^S[X]$ is congruent to $(X-\alpha_v)(X-\beta_v)$ modulo $\mathfrak{m}$, and the latter is the Hecke polynomial of $\overline{g} \in S_2(\Gamma,\mathbb{F})$ (note 
that $\overline{g} \equiv g \mod(\varpi)$, where $\varpi$ is a fixed uniformizer of $\mathcal{O}$). By the theory of old forms, we know that there is an $\overline{h} \in S_2(\Gamma_0(Q),\mathbb{F})$ that has the 
same $T_\ell$-eigenvalues and $S_\ell$-eigenvalues as $\overline{g}$ for all primes $\ell \notin S \cup Q$ and that $U_v\overline{g}=\alpha_v\overline{g}$ for all $v \in Q$. Thus, by choosing any 
lift $\tilde{\alpha}_v \in \mathcal{O}$ of $\alpha_v$ for all $v \in Q$, we get a maximal ideal $\mathfrak{m}_Q:=\langle \mathfrak{m}, \{U_v - \tilde{\alpha}_v \}_{v \in Q} \rangle$ of $\T^{S \cup Q, Q}$ and both maximal 
ideals $\mathfrak{m} \subseteq \T^{S \cup Q}$ and $\mathfrak{m}_Q \subseteq \T^{S \cup Q, Q}$ are in the support of $H^1(Y_0(Q),\mathcal{O})$ and $H_1(Y_0(Q),\mathcal{O})$. Also, we again have the duality between 
homology and cohomology, after localizing at either $\mathfrak{m}$ or $\mathfrak{m}_Q$. Note also that since $\T^{S \cup Q}(\Gamma_0(Q))$ and $\T^{S \cup Q, Q}(\Gamma_0(Q))$ are finite $\mathcal{O}$-algebras, 
so $\T^{S \cup Q}(\Gamma_0(Q))_{\mathfrak{m}}$ is a complete Noetherian local ring; hence the localization of $\T^{S \cup Q, Q}(\Gamma_0(Q))$ at $\mathfrak{m} \subseteq \T^{S \cup Q}(\Gamma_0(Q))$ is a 
complete semilocal ring, and thus it is a product of its local rings of which $\T^{S \cup Q,Q}(\Gamma_0(Q))_{\mathfrak{m}_Q}$ is one. In particular, $H_1(Y_0(Q),\mathcal{O})_{\mathfrak{m}_Q}$ is 
a direct summand of $H_1(Y_0(Q),\mathcal{O})_{\mathfrak{m}}$. Note that similar statements hold, when we replace $\Gamma_0(Q)$ by $\Gamma_Q$.

\begin{proposition} \label{prop: first iso of homologies}
The natural map $H_1(Y_0(Q),\mathcal{O}) \rightarrow H_1(Y,\mathcal{O})$ induces an isomorphism $H_1(Y_0(Q),\mathcal{O})_{\mathfrak{m}_Q} \simeq H_1(Y,\mathcal{O})_{\mathfrak{m}}$ of $\T^{S \cup Q}$-modules.
\end{proposition}

\begin{proposition} \label{prop: second iso of homologies}
The homology group $H_1(Y_Q,\mathcal{O})_{\mathfrak{m}_Q}$ is a free $\mathcal{O}[\Delta_Q]$-module and the natural map $H_1(Y_Q,\mathcal{O})_{\mathfrak{m}_Q} \rightarrow H_1(Y_0(Q),\mathcal{O})_{\mathfrak{m}_Q}$ induces an isomorphism from 
the $\Delta_Q$-coinvariants of $H_1(Y_Q,\mathcal{O})_{\mathfrak{m}_Q}$ to $H_1(Y_0(Q),\mathcal{O})_{\mathfrak{m}_Q}$, i.e. $H_0(\Delta_Q, H_1(Y_Q,\mathcal{O})_{\mathfrak{m}_Q}) \simeq H_1(Y_0(Q),\mathcal{O})_{\mathfrak{m}_Q}$ (by the $\Delta_Q$-coinvariant of an $\mathcal{O}[\Delta_Q]$-module $M$, we mean $H_0(\Delta_Q,M)=M/\mathfrak{a}M$, where $\mathfrak{a}$ is the augmentation ideal).
\end{proposition}

By combining the two above propositions, we get the following corollary:

\begin{corollary} \label{cor: third iso of homologies}
The natural map $H_1(Y_Q,\mathcal{O})_{\mathfrak{m}_Q} \rightarrow H_1(Y,\mathcal{O})_{\mathfrak{m}}$ induces an isomorphism from the $\Delta_Q$-coinvariants of $H_1(Y_Q,\mathcal{O})_{\mathfrak{m}_Q}$ to $H_1(Y,\mathcal{O})_{\mathfrak{m}}$.
\end{corollary}

Recall that if we have a global deformation condition $\mathcal{S}$ with universal deformation ring $R_{\mathcal{S}}$, then for a finite set of Taylor--Wiles primes $Q$ we have a global deformation condition $\mathcal{S}_Q$ with universal deformation ring $R_{\mathcal{S}_Q}$ which is an $\mathcal{O}[\Delta_Q]$-algebra such that $R_{\mathcal{S}_Q}/\mathfrak{a}_Q R_{\mathcal{S}_Q} \simeq R_{\mathcal{S}}$.
Note that we also have the Galois representations $\rho_{\mathfrak{m}}:G_{\mathbb{Q},S} \rightarrow GL_2(\T^{S}(\Gamma)_{\mathfrak{m}})$ and $\rho_{\mathfrak{m}_Q}:G_{\mathbb{Q},S} \rightarrow GL_2(\T^{S \cup Q, Q}(\Gamma)_{\mathfrak{m}_Q})$. If 
they are of type $\mathcal{S}$ and $\mathcal{S}_Q$, respectively, then we have the following commutative diagram:

\begin{equation*}
\begin{tikzcd}
R_{\mathcal{S}_Q} \arrow[r , bend left] \arrow[d , twoheadrightarrow]
& H_1(Y_Q,\mathcal{O})_{\mathfrak{m}_Q} \arrow[d , twoheadrightarrow] \\
R_{\mathcal{S}} \arrow[r , bend left]
&  H_1(Y,\mathcal{O})_{\mathfrak{m}}
\end{tikzcd}
\end{equation*}

where both vertical maps are ``$\text{mod} \: \mathfrak{a}_Q$'' maps.

\section{Taylor-Wiles Patching Method and $R=\mathbb{T}$}
\label{sec:3}

\subsection{Minimal Case}

Fix a newform $g \in S_2(\Gamma_1(N),\overline{\mathbb{Q}}_p)$ and let $\eta$ be its nebentypus. Let $\overline{\rho}:=\overline{\rho}_g:G_{\mathbb{Q}} \rightarrow GL_2(\overline{\mathbb{F}}_p)$ be the associated $\text{mod} \: p$ Galois representation to $g$.
There is a finite extension $\mathbb{F}$ of $\mathbb{F}_p$ which contains the image of $\overline{\rho}$.
Assume that $\mathbb{F}$ is sufficiently large such that for all $\sigma \in G_{\mathbb{Q}}$, the eigenvalues of $\overline{\rho}(\sigma)$ are in $\mathbb{F}$.

From now, we assume the following for $\overline{\rho}=\overline{\rho}_g$:
\begin{itemize}
\item[(i)] $p>2$ and $p \nmid N$.
\item[(ii)] $\overline{\rho}|_{G_{\mathbb{Q}(\zeta_p)}}$ is absolutely irreducible with enormous image (if $p \geq 7$, then the condition on the image holds by Theorem \ref{th: classification of enormous subgroups}).
\item[(iii)] \: $N$ is square-free, $\overline{\rho}$ is ramified at all primes dividing $N$ and $\eta$ has prime-to-$p$ order. Equivalently, we assume that $\overline{\rho}$ is modular of weight $2$ and level $N(\overline{\rho})=$ \textit{Artin conductor} such that $N(\overline{\rho})$ is square-free.
\item[(iv)] \: $\overline{\rho}_{G_{\mathbb{Q}_p}} \simeq
\begin{bmatrix}
\overline{\chi}_1 & *\\
0 & \overline{\chi}_2
\end{bmatrix}$
with $\overline{\chi}_1|_{I_p}=1$ and $\overline{\chi}_2|_{I_p}=\overline{\varepsilon}_p^{-1}$, where $\overline{\varepsilon}_p$ is the mod-$p$ cyclotomic character.
\end{itemize}

Now,we define a global deformation condition $\mathcal{S}=(\overline{\rho},S,\chi,\mathcal{O},\{D_v\}_{v \in S})$ by letting:
\begin{itemize}
\item[(1)] $S=\{\ell: \ell|N\} \cup \{p\}$.
\item[(2)] $\chi=\eta \varepsilon_p$.
\item[(3)] $D_v= \begin{cases} D_v^{min} & v|N \\
D_v^{ord} & v=p
\end{cases} $.
\end{itemize}

\begin{remark} \label{rem: minimal case is restrective}
The assumption $D_v= \begin{cases} D_v^{min} & v|N \\
D_v^{ord} & v=p
\end{cases} $ for defining the above global deformation problem $\mathcal{S}$ is \textit{restrictive} for \textit{modularity lifting} purposes, but still has its own interesting consequences, e.g. \textit{modularity lifting in the minimal case} (for the definitions of $D_v^{ord}$ and $D_v^{min}$ look at \ref{subsub: local deformation}).
\end{remark}

Let $\Gamma \supseteq \Gamma_1(N)$ be the kernel of the composition $\Gamma_0(N) \rightarrow (\mathbb{Z}/N\mathbb{Z})^\times \rightarrow \overline{\mathbb{Q}}_p^\times$, where the right map in the composition is $\eta$. Also, assume $\Gamma$ is torsion-free. Let $\mathfrak{m} \subseteq \T^{S}$ be the maximal ideal that corresponds to $\overline{\rho}$. Then, we have the following important theorem:

\begin{theorem} \label{th: R to T is surjective}
The Galois representation $\rho_{\mathfrak{m}}:G_{\mathbb{Q},S} \rightarrow GL_2(\T^S(\Gamma)_{\mathfrak{m}})$ lifting $\overline{\rho}$ is of type $\mathcal{S}$. Hence, there is a map $R_{\mathcal{S}} \rightarrow \T^S(\Gamma)_{\mathfrak{m}}$ in $\textbf{CNL}_{\mathcal{O}}$. Also, this map is surjective.
\end{theorem}

\begin{remark}
The goal is to show that the above surjection is indeed an isomorphism.
\end{remark}

\subsubsection{Patching}

Note that we will use \textit{Diamond's} modification of the original patching argument \cite{Diamond}.

We continue to assume the assumptions that we made in this section. Let $(Q, \{\alpha_v\}_{v \in Q})$ be a Taylor--Wiles datum. Let $\T^{S \cup Q}(\Gamma_Q)_{\mathfrak{m}_Q}$ be the subalgebra of $\End_{\mathcal{O}}(H_1(Y,\mathcal{O})_{\mathfrak{m}_Q})$ that is generated by $T_\ell$ and $S_\ell$ for all primes $\ell \notin S \cup Q$, and by $\langle \delta \rangle$ for all $\delta \in \Delta_Q$. Then, we have the following theorem:

\begin{theorem} \label{th: existence of rho_Q}
There exists a continuous Galois representation $\rho_Q:G_{\mathbb{Q},S \cup Q} \rightarrow GL_2(\T^{S \cup
Q}(\Gamma_Q)_{\mathfrak{m}_Q})$ such that:
\begin{itemize}
\item[(i)] For any $\ell \notin S \cup Q$, the characteristic polynomial of $\rho_Q(\Frob{\ell})$ is given by $X^2-T_\ell X +\ell S_\ell$.
\item[(ii)] For any $v \in S$, $\rho_Q|_{G_{\mathbb{Q}_v}} \in D_v$.
\item[(iii)] \: For any $v \in Q$, $\rho_Q|_{I_v} \simeq 1 \oplus \chi_v$, where $\chi_v \circ Art_{\mathbb{Q}_v}(\delta)=\langle \delta \rangle$.
\end{itemize}
\end{theorem}

So, by applying Theorem \ref{th: R to T is surjective} to the previous Theorem, we find that there exists a surjection $R_{\mathcal{S}_Q} \twoheadrightarrow \T^{S \cup Q}(\Gamma_Q)_{\mathfrak{m}}$. Also, $H_1(Y_Q,\mathcal{O})_{\mathfrak{m}_Q}$ has an $R_{\mathcal{S}_Q}$-module structure which is compatible with its $\mathcal{O}[\Delta_Q]$-module structure.

\begin{proposition} \label{prop: there exist R_infty and M_infty}
There is a non-negative integer $q$, the $\textbf{CNL}_{\mathcal{O}}$-algebra $R_\infty:=\mathcal{O}\llbracket X_1, \cdots , X_q\rrbracket $ and a finitely generated $R_\infty$-module $M_\infty$ such that the following diagram is commutative and satisfying the following properties ($S_\infty:=\mathcal{O}\llbracket Y_1, \cdots , Y_q\rrbracket$):
\begin{equation} \label{eq: minimal case}
\begin{tikzcd}
S_\infty \arrow[r, ]
& R_\infty \arrow[r, bend left] \arrow[d, twoheadrightarrow]
& M_\infty \arrow[d,twoheadrightarrow ] \\
&
R:=R_{\mathcal{S}} \arrow[r, bend left]
& M:=H_1(Y,\mathcal{O})_{\mathfrak{m}}
\end{tikzcd}
\end{equation}

\begin{itemize}
\item[(1)] The $R_\infty$-module $M_\infty$ is a finite free $S_\infty$-module.
\item[(2)] We have the surjections $R_\infty \twoheadrightarrow R$ and $M_\infty \twoheadrightarrow M$ such that kernel of the first map is contained in $\mathfrak{a}R_\infty$ and kernel of the second map is equal to $\mathfrak{a}M_\infty$, where $\mathfrak{a}:=\langle Y_1, \cdots , Y_q \rangle \subseteq S_\infty$ is the augmentation ideal.
\end{itemize}
\end{proposition}

By using this Proposition, one can prove the following $R=\T$ statement:

\begin{theorem} \label{R=T minimal case}
The surjection $R_{\mathcal{S}} \twoheadrightarrow \T^S(\Gamma)_{\mathfrak{m}}$ (look at Theorem \ref{th: R to T is surjective}) is, in fact, an isomorphism of local complete intersection rings.
\end{theorem}

\begin{proof}

By Proposition \ref{prop: there exist R_infty and M_infty}, we have that $M_\infty$ is a finite free $S_\infty$-module and its $S_\infty$-module structure factors through $R_\infty$. Thus we have:
\begin{equation*}
1+q \geq \dim R_\infty \geq \dim _{R_\infty}M_\infty \geq \depth _{R_\infty}M_\infty \geq \depth _{S_\infty}M_\infty = \dim S_\infty =1+q,
\end{equation*}
so all above inequalities should be equalities (note that the equality $\depth _{S_\infty}M_\infty = \dim S_\infty$ follows from the fact that $M_\infty$ is a finite free $S_\infty$-module).
Since $R_\infty$ is regular, then by \textit{Serre's theorem}, $M_\infty$ has a projective resolution of finite length. Thus, we can use the \textit{Auslander--Buchsbaum formula}:
\begin{equation*}
\pd _{R_\infty} M_\infty = \depth R_\infty - \depth M_\infty = (1+q) - (1+q) =0,
\end{equation*}
where $\pd _{R_\infty} M_\infty$ is the projective dimension of the $R_\infty$-module $M_\infty$. Therefore, $M_\infty$ is a projective $R_\infty$ module, hence it is free because $R_\infty$ is local. 
Once again, by Proposition \ref{prop: there exist R_infty and M_infty} 
, we find that $M \simeq M_\infty/\mathfrak{a}M_\infty$ is a free module over $R \simeq R_\infty/\mathfrak{a}R_\infty$. 
But, the $R$-module structure on $M$ is defined via the surjection $R=R_\mathcal{S} \rightarrow \T^S(\Gamma)_{\mathfrak{m}}$ (look at Theorem \ref{th: R to T is surjective}). 
If $0 \neq r \in R$ be in the kernel of this surjection map, then $r \in \Ann _R (M)$ which is impossible since $M$ is free over $R$. Thus we get $R=R_\mathcal{S} \simeq \T^S(\Gamma)_{\mathfrak{m}}$. 
Moreover, these rings are complete intersection rings because we have a presentation:
\begin{equation*}
R=R_{\mathcal{S}} \simeq R_\infty /\mathfrak{a} = \mathcal{O}\llbracket X_1,\cdots,X_q\rrbracket /\langle Y_1, \cdots Y_q\rangle
\end{equation*}
and $\dim R = \dim \T^S(\Gamma)_{\mathfrak{m}} = 1$.
\end{proof}

Now, let us see how one constructs $M_\infty$ and the surjections $R_\infty \twoheadrightarrow R$ and $M_\infty \twoheadrightarrow M$ (as inverse limits of modules and maps).

\begin{definition} \label{def: patching datum}
Put $q=h^1_{\mathcal{S}^\perp}(ad^0(\overline{\rho})(1))$, $S_\infty=\mathcal{O}\llbracket \mathbb{Z}_p^q\rrbracket =\mathcal{O}\llbracket Y_1, \cdots, Y_q\rrbracket $. For any positive integer $N$, let we put:
\begin{itemize}
\item[(i)] $\mathfrak{a}_N:=\Ker(S_\infty \twoheadrightarrow \mathcal{O}[(\mathbb{Z}/p^N\mathbb{Z})^q])$.
\item[(ii)] $S_N:=S_\infty/\langle \varpi^N, \mathfrak{a}_N \rangle$ (recall that $\varpi$ is a fixed uniformizer of $\mathcal{O}$).
\item[(iii)] \: $\mathfrak{d}_N:=\langle \varpi^N, \Ann_R(M)^N \rangle$.
\end{itemize}
We define a \textit{patching datum of level $N$} to be a triple $(f,X,g)$, where:
\begin{itemize}
\item[(1)] $f:R_\infty \rightarrow R/\mathfrak{d}_N$ is a surjection in $\textbf{CNL}_{\mathcal{O}}$.
\item[(2)] $X$ is an $R_\infty \tensor{\mathcal{O}} S_N$-module which is finite and free over $S_N$, such that:
\begin{itemize}
\item[(i)] $\im(S_N \rightarrow \End_{\mathcal{O}}(X)) \subseteq \im(R_\infty \rightarrow \End_{\mathcal{O}}(X))$.
\item[(ii)] $\im(\mathfrak{a} \rightarrow \End_{\mathcal{O}}(X)) \subseteq \im(\Ker(f) \rightarrow \End_{\mathcal{O}}(X))$.
\end{itemize}
\item[(3)] $g:X/\mathfrak{a} \rightarrow M/\langle \varpi^N \rangle$ is an isomorphism of $R_\infty$-modules.
\end{itemize}
We say that two patching data $(f,X,g)$ and $(f^\prime,X^\prime,g^\prime)$ of level $N$ are isomorphic, if $f=f^\prime$ and there exists an isomorphism $X \simeq X^\prime$ of $R_\infty \tensor{\mathcal{O}} S_N$-modules which is compatible with $g$ and $g^\prime$.
\end{definition}

\begin{remark}
An important fact is that there are only finitely many isomorphic classes of patching data of a fixed level $N$.
\end{remark}

Note that if $M \geq N$ be two positive integers and if $D=(f,X,g)$ is a patching datum of level $M$, then $D \: \text{mod} \: N:=(f \: \text{mod} \: \mathfrak{d}_N, X \tensor{S_M} S_N, g \tensor{S_M} S_N)$ is a patching datum of level $N$.

Recall that by Propositon \ref{prop: existence of TW primes Q_N}, for each positive integer $N$, we can choose a Taylor--Wiles datum $(Q_N,\{\alpha_v\}_{v \in Q_N})$ of level $N$ such that for all $N$ we have :
\begin{itemize}
\item[(i)] $|Q_N|=q$.
\item[(ii)] $h^1_{\mathcal{S}^\perp_{Q_N}}(ad^0(\overline{\rho})(1))=0$.
\end{itemize}

By what we have seen until now, for any positive integer $N$ we can define a patching datum of level $N$ by $D_N:=(f_N,X_N,g_N)$, with:
\begin{itemize}
\item[(1)] $f_N:R_\infty \rightarrow R_{\mathcal{S}_{Q_N}} \rightarrow R \rightarrow R/\mathfrak{d}_N$, where the map $R_\infty= \mathcal{O}\llbracket X_1, \cdots , X_q\rrbracket \twoheadrightarrow R_{\mathcal{S}_{Q_N}}$ comes from the fact that the $\mathcal{O}$-relative tangent space of $R_{\mathcal{S}_{Q_N}}$ has dimension $q:=h^1_{\mathcal{S}_{Q_N}}(ad_0(\overline{\rho}))$.
\item[(2)] $X_N:=H_1(Y_{Q_N},\mathcal{O})_{\mathfrak{m}_{Q_N}} \tensor{S_\infty} S_N$.
\item[(3)] $g$ is induced from the isomorphism between $H_1(Y,\mathcal{O})_{\mathfrak{m}}$ and the $\Delta_{Q_N}$-coinvariants of $H_1(Y_{Q_N},\mathcal{O})_{\mathfrak{m}_{Q_N}}$ (look at Corollary \ref{cor: third iso of homologies}).
\end{itemize}

Then, for positive integers $M \geq N$ and a patching datum of level $M$, we have a patching datum of level $N$ by defining it as $D_{M,N}:=D_M \: \text{mod} \: N =(f_{M,N},X_{M,N},g_{M,N})$. Now, since for any positive integer $N$, there are infinitely many $M \geq N$ and only finitely many isomorphism classes of patching data of level $N$, we can find a subsequence $(M_i,N_i)_{i \geq 1}$ with $M_i \geq N_i$ and $N_{i+1}>N_i$ such that $D_{M_{i+1},N_{i+1}} \: \text{mod} \: N_i \simeq D_{M_i,N_i}$. Then:
\begin{itemize}
\item[(i)] The $R_\infty$-module $M_\infty$ is defined as $\varprojlim X_{M_i}$.
\item[(ii)] The map $R_\infty \twoheadrightarrow R$ is defined as $\varprojlim f_{M_i,N_i}$.
\item[(iii)] \: The map $M_\infty \twoheadrightarrow M$ is defined as $\varprojlim g_{M_i,N_i}$.
\end{itemize} 

\begin{remark}
Let us mention a motivation behind the patching method. In some sense, \textit{modularity} is a $GL_2$ version of the \textit{Iwasawa main conjecture}, which considered as a $GL_1$ problem (nowadays we have a $GL_2$ version of Iwasawa main conjecture itself). 
In fact, in Iwasawa theory we have a good module to work with, namely the inverse limit of the $p$-parts of the class groups of the number fields in the tower of our $\mathbb{Z}_p$-extension. 
Note that, in this case, the $p$-parts of class groups trivially make an inverse system. In our situation, the patching method construct a good module $M_\infty$ and the maps $R_\infty \twoheadrightarrow R$ and $M_\infty \twoheadrightarrow M$. 
In the patching method, we need a compatible system of patching data (as an analog of the system of $p$-parts of class groups), where we change the level via Taylor--Wiles primes, hence we need compatibility properties in the deformation problems attached to Taylor--Wiles primes. 
This is the reason why we had study Taylor--Wiles primes and the properties of corresponding deformation problems in detail. Recall that for the ramification, by definition we know adding Taylor--Wiles primes does not change the ramified primes in our deformation problem. 
Note that, this also is like the Iwasawa theoretic context, namely ramified primes are the same in the our $\mathbb{Z}_p$-tower (after a finite layer).
\end{remark}

Now, let us state (and prove!) \textit{a modularity lifting theorem in the minimal case}, using our $R=\T$ theorem (Theorem \ref{R=T minimal case}):

\begin{theorem} \label{th: minimal modularity lifting}
Let $p$ be an odd prime and $\rho:G_{\mathbb{Q}} \rightarrow GL_2(\overline{\mathbb{Q}}_p)$ be a continuous irreducible Galois representation satisfying the following conditions:
\begin{itemize}
\item[(1)] $\rho$ is unramified outside a finite set of primes.
\item[(2)] $\rho|_{G_{\mathbb{Q}_p}} \simeq \begin{bmatrix}
\chi_1 & *\\
0 & \chi_2
\end{bmatrix}
with \chi_1|_{I_p}=1$ and $\chi_2|_{I_p}=\varepsilon_p^{-1}$, where $\varepsilon_p$ is the $p$-adic cyclotomic character.
\item[(3)] $\overline{\rho}|_{G_{\mathbb{Q}(\zeta_p)}}$ is absolutely irreducible with enormous image.
\item[(4)] For all $\ell \neq p$ at which $\rho$ is ramified, we have either:
\begin{itemize}
\item[(i)] $\rho|_{I_\ell} \simeq 1\oplus \theta$ with $\theta(I_\ell) \simeq \overline{\theta}(I_\ell)$, or
\item[(ii)] $\rho|_{I_\ell}$ is isomorphic to the image of $\rho$ in the set of matrices of the form $\begin{bmatrix}
1 & *\\
0 & 1
\end{bmatrix}$ and $\overline{\rho}(I_\ell) \neq 1$;
\end{itemize}
and for $p$ we have:
\begin{itemize}
\item $\overline{\rho}|_{G_{\mathbb{Q}_p}} \simeq \begin{bmatrix}
\overline{\chi}_1 & *\\
0 & \overline{\chi}_2
\end{bmatrix}$ with $\overline{\chi}_1 \overline{\chi}_2 ^{-1} \neq 1, \overline{\varepsilon}$.
\end{itemize}
\item[(5)] $\overline{\rho} \simeq \overline{\rho}_g$ for some $g \in S_2(\Gamma_1(N),\overline{\mathbb{Q}}_P)$, with $N=\prod \ell$ where $\ell \neq p$ runs over all primes at which $\rho$ is ramified.
\end{itemize}
Then, $\rho \simeq \rho_f$ for some Hecke-eigenform $f \in S_2(\Gamma_1(N),\overline{\mathbb{Q}}_p)$.
\end{theorem}

In fact, from the assumptions of the Theorem, we can find an $\mathcal{O}$-algebra homomorphism $R_{\mathcal{S}} \rightarrow \overline{\mathbb{Q}}_p$ with $\mathcal{S}$ as in this section. 
Now, $R_{\mathcal{S}} \simeq \T^S(\Gamma)_{\mathfrak{m}}$ (Theorem \ref{R=T minimal case}), for $S=\{\ell: \ell|N\} \cup \{p\}$, implies that there exists an $\mathcal{O}$-algebra homomorphism $\lambda: \T^S(\Gamma)_{\mathfrak{m}} \rightarrow \overline{\mathbb{Q}}_p$ which 
is an eigensystem of (and hence, equivalent to) some Hecke-eigenform $f \in S_2(\Gamma_1(N),\overline{\mathbb{Q}}_p)$ (since the characteristic polynomial of $\rho(\Frob{\ell})$ is given by $X^2-\lambda(T_\ell)X+\ell \lambda(S_\ell)$).

\subsection{Non-minimal Case}

Even though we are happy to have proved a modularity lifting theorem in the minimal case, it is not enough to deduce the \textit{Shimura--Taniyama--Weil (STW) Conjecture}, even in the semistable case. For deducing \textit{STW} in 
the semistable case, we need a \textit{non-minimal modularity lifting theorem}, which itself follows from an $R^{\red}=\T$ theorem. In fact, the fourth condition in the previous modularity lifting theorem (Theorem \ref{th: minimal modularity lifting}) is restrictive. There are (at least) two ways to get rid of it:
\begin{itemize}
\item[(i)] Wiles' method \cite{Wiles}: numerical criterion. Note that it is hard to generalize it. 
\item[(ii)] Kisin's method \cite{Kisin}: presenting global deformation rings as algebras over local lifting rings.
\end{itemize}
We will try to give a sketch of Kisin's method.
\\

Let us continue to assume that $\overline{\rho}$ is modular, i.e. $\overline{\rho}=\overline{\rho}_g$ for some $g \in S_2(\Gamma_1(N),\overline{\mathbb{Q}}_p)$ and $\overline{\rho}|_{G_{\mathbb{Q}(\zeta_p)}}$ is 
absolutely irreducible with enormous image, but let us drop the \textit{minimality hypothesis}, so maybe the level of $\Gamma=\Gamma_1(N)$ (which is equal to $N$) be non-square-free and lifts of $\overline{\rho}$ ramified 
at some primes for which $\overline{\rho}$ itself is unramified. So, we can make a global deformation condition $\mathcal{S}=(\overline{\rho},S,\chi,\mathcal{O},\{D_v\}_{v \in S})$,
 $D_v \in D^{\square,\chi}_{\overline{\rho}|_{G_{\mathbb{Q}_v}}}$ such that we can 
prove $\rho_{\mathfrak{m}}: G_{\mathbb{Q}} \rightarrow GL_2(\T^S(\Gamma)_{\mathfrak{m}})$ is of type $\mathcal{S}$ and we expect all deformations of $\overline{\rho}$ of type $\mathcal{S}$ come 
from $\T^S(\Gamma)_{\mathfrak{m}}$. Note that also we assume for any $v \in S$, the ring $R_v$ which represents $D_v$ is $\mathcal{O}$-flat. Furthemore, we have $\dim R_v=
\begin{cases}
4 & v \neq p \\
5 & v=p
\end{cases}$.

We consider frames at $T=S$ and put $|S|=s$. Let $R_{\mathcal{S}}^{loc}:= \hat{\tensor{\mathcal{O}, v \in S}} R_v$ is $\mathcal{O}$-flat of dimension $2+3s$. 
Also, recall that $R^S_{\mathcal{S}_Q} \simeq R_{\mathcal{S}_Q} \hat{\tensor{\mathcal{O}}} \Omega$ where $\Omega=\mathcal{O}\llbracket Z_1, \cdots, Z_{4s-1}\rrbracket$ (look at the explanation just before the Diagram \eqref{eq: non-min}). Then, we have the following important proposition:

\begin{proposition} \label{prop: M_infty in non-minimal case}
There is a non-negative integer $q$, the $\textbf{CNL}_{\mathcal{O}}$-algebra $R_\infty:=R_{\mathcal{S}}^{loc}\llbracket X_1, \cdots , X_g\rrbracket$ and a finitely generated $R_\infty$-module $M_\infty$ such that the following diagram is commutative and satisfying the following properties ($S_\infty:=\Omega\llbracket Y_1, \cdots , Y_q\rrbracket$):
\begin{equation} \label{eq: minimal case}
\begin{tikzcd}
S_\infty \arrow[r, ]
& R_\infty \arrow[r, bend left] \arrow[d, twoheadrightarrow]
& M_\infty \arrow[d,twoheadrightarrow ] \\
&
R:=R_{\mathcal{S}} \arrow[r, bend left]
& M:=H_1(Y,\mathcal{O})_{\mathfrak{m}}
\end{tikzcd}
\end{equation}

\begin{itemize}
\item[(1)] The $R_\infty$-module $M_\infty$ is a finite free $S_\infty$-module.
\item[(2)] We have the surjections $R_\infty \twoheadrightarrow R$ and $M_\infty \twoheadrightarrow M$ such that kernel of the first map is contained in $\mathfrak{a}R_\infty$ and kernel of the second map is equal to $\mathfrak{a}M_\infty$, where $\mathfrak{a}:=\langle Z_1 , \cdots ,Z_{4s-1}, Y_1, \cdots , Y_q \rangle \subseteq S_\infty$ is the augmentation ideal.
\item[(3)] We have $\dim S_\infty=\dim R_\infty$, i.e. $4s+q=g+2+3s$ which means $s+q=g+2$.
\end{itemize}
\end{proposition}

Note that in the above proposition, the patching datum is defined similar to the previous case.

\begin{proposition} \label{prop: R^red=T}
If $\Supp_{R_\infty}(M_\infty)=\Spec(R_\infty)$, then $\Supp_R(M)=\Spec(R)$ and the surjective map $R \twoheadrightarrow \T^S(\Gamma)_{\mathfrak{m}}$ has nilpotent kernel, 
hence the map $R^{\red} \rightarrow \T^S(\Gamma)_{\mathfrak{m}}$ is an isomorphism.
\end{proposition}

Note that, $R^{\red} \simeq \T^S(\Gamma)_{\mathfrak{m}}$ is good enough for our modularity lifting purposes. So, the problem is to show that $M_\infty$ has full support in $\Spec(R_\infty)$. There are (at least) two ways to do this:
\begin{itemize}
\item[(i)] By using \textit{Ihara's Lemma} \cite{Wiles}; or
\item[(ii)] By using \textit{Taylor's Ihara avoidance} trick \cite{Taylor1}.
\end{itemize}

We do not go into this. We end the chapter by stating a \textit{non-minimal modularity lifting} result, which follows from our $R^{\red}=\T$ (Proposition \ref{prop: R^red=T}); and some remarks. 

\begin{theorem}
Let $p \geq 5$ be a prime and let $\rho:G_{\mathbb{Q}} \rightarrow GL_2(\overline{\mathbb{Q}}_p)$ be a continuous irreducible Galois representation satisfying the following:
\begin{itemize}
\item[(1)] $\rho$ is unramified outside a finite set of primes.
\item[(2)] $\rho|_{G_{\mathbb{Q}_p}}$ satisfies some $p$-adic Hodge theoretic conditions.
\item[(3)] $\overline{\rho}|_{G_{\mathbb{Q}(\zeta_p)}}$ is absolutely irreducible with enormous image.
\item[(4)] $\overline{\rho} \simeq \overline{\rho}_g$ for a modular form $g \in S_2(\Gamma_1(N),\overline{\mathbb{Q}}_p)$ with $p \nmid N$.
\end{itemize}
Then, $\rho \simeq \rho_f$ for a modular form $f \in S_2(\Gamma_1(N),\overline{\mathbb{Q}}_p)$.
\end{theorem}

\begin{remark}
Note that, in the above modularity lifting theorem, we make no assumption on the ramification of $\rho$ and on the level of $g$ at primes different from $p$.
\end{remark}

\begin{remark}
There is another method for patching, which is due to \textit{Peter Scholze} \cite{Scholze}, \cite{Taylor2}.
\end{remark}

\section*{Acknowledgment}
We would like to thank the anonymous referee for carefully reading the first draft and for their valuable comments and suggestions.

\bibliographystyle{amsplain}

\end{document}